\documentclass[12pt]{article}
\usepackage{amsmath, graphicx, amsfonts,amssymb, calrsfs}
\usepackage{amsfonts,mathrsfs, color, amsthm}

\usepackage{accents}
\usepackage{enumitem}
\usepackage{indentfirst}
\usepackage{fancyhdr}

\def\sphere{\mathbb{S}^{n-1}}

\def\Rn{{\mathbb R^n}}
\def\R{\mathbb{R}}

\newtheorem{theorem}{Theorem}[section]
\newtheorem{lemma}{Lemma}[section]
\newtheorem{remark}{Remark}[section]
\newtheorem{proposition}{Proposition}[section]
\newtheorem{corollary}{Corollary}[section]

\newtheorem{definition}{Definition}[section]

\def\bpf{\begin{proof}}
\def\epf{\end{proof}}
\def\be{\begin{equation}}
\def\ee{\end{equation}}
\def\bea{\begin{eqnarray}}
\def\eea{\end{eqnarray}}
\def\bt{\begin{theorem}}
\def\et{\end{theorem}}
\def\bl{\begin{lemma}}
\def\el{\end{lemma}}
\def\br{\begin{remark}}
\def\er{\end{remark}}
\def\bc{\begin{corollary}}
\def\ec{\end{corollary}}
\def\bd{\begin{definition}}
\def\ed{\end{definition}}
\def\bp{\begin{proposition}}
\def\ep{\end{proposition}}

\numberwithin{equation}{section}

\begin{document}
\title{The $L_q$ Minkowski problem for $\mathbf{p}$-harmonic measure
\footnote{ These authors contributed equally: Hai Li, Longyu Wu. $^\ddagger$Corresponding author: email: bczhu@snnu.edu.cn}}

\author{Hai Li$^\dagger$, Longyu Wu$^\dagger$, Baocheng Zhu$^{\dagger,\ \ddagger}$}
\date{\quad}

\maketitle
\begin{abstract}
In this paper, we consider an extremal problem associated with the solution to a boundary value problem.
Our main focus is on establishing a variational formula for a functional related to the $\mathbf{p}$-harmonic measure,
from which a new measure is derived. This further motivates us to study the Minkowski problem for this new measure.
As a main result, we prove the existence of solutions to the $L_q$ Minkowski problem associated with the $\mathbf{p}$-harmonic
measure for $0<q<1$ and $1<\mathbf{p}\ne n+1$.

\end{abstract}

\section{Introduction}\label{sect:1}

The $L_q$ Minkowski problem is one of the most important contents in convex geometry.
It can be stated as: For any given $q\in \R$ and a finite nonzero Borel measure $\mu$
on the unit sphere $\sphere$ in $\R^n$, whether there exists a convex body whose $L_q$
surface area measure is the given measure $\mu$. When $q=1$, the $L_q$ Minkowski problem
reduces to the classical one, which dates back to the early works
by Minkowski and was developed further by Aleksandrov, Fenchel and Jessen.
The $L_q$ Minkowski problem for $q>1$ was first studied by Lutwak \cite{L93}. Since then,
this problem has received significant attention, leading to remarkable progress (see e.g., \cite{HS04, HZ05, LZ04, U03}).
When $q<1$, the problem is more challenging (see e.g., \cite{C06, CW06, DZ12, JZ16, LW13, Z15}).
Particularly for $q=0$,  it becomes the logarithmic Minkowski problem (see e.g.,
\cite{BZ13, CL22, LX24, S02, S03, TX23, Z14}). For more progress on the $L_q$ Minkowski problem,
we refer to \cite{CL20, HX15, M24} and the references therein. It is well known that
the solutions to the $L_q$ Minkowski problem are key ingredients in the rapidly developing
$L_q$ Brunn-Minkowski theory of convex bodies. For instance, they have played an important role
in establishing affine Sobolev inequalities (see e.g., \cite{CZ09, HS09, LZ02, Z99}).

Along with the rapid development of the Brunn-Minkowski theory, the Minkowski
problem has been greatly enriched. Examples include the Minkowski problem for
the dual curvature measure \cite{HZ16, LW20}, the Gaussian
surface area measure \cite{CZ23, FX23, HZ21}, the chord measure \cite{GZ24, LZ24++, XZ23},
and the Minkowski problem for unbounded closed convex sets \cite{LZ24+, S18, S24, YZ23},
as well as for log-concave functions \cite{CK15, FY22, R22}. These problems are well-known
for their close relationships among convex geometry, integral geometry,
differential geometry, and PDEs. Jerison systematically integrated the Brunn-Minkowski theory with
potential theory and the regularity theory of fully nonlinear equations. In his earlier works \cite{J89, J91},
he first studied the Minkowski problem for harmonic measure. Later, in another paper \cite{J96},
he examined a similar problem for electrostatic capacity. Jerison's contributions sparked significant research
into Minkowski problems. A notable example of ongoing research is the study of the Minkowski problem
for $\mathbf{p}$-capacity by Colesanti et al. \cite{CZ15}.
Recently, this problem has been extended to the $L_q$
case \cite{ZX20}. In fact, such kind of Minkowski problem
is closely related to a boundary value problem. More examples
of Minkowski problems associated with the boundary value problems
include those for capacity \cite{AV22, HZ18, LH23, X20, XX19} and for
torsional rigidity \cite{CF10, HZ23, LZ20}.

Let $K$ be a bounded convex domain with boundary $\partial K$ and $N$
be a neighborhood of $\partial K$. In this paper, we consider the
following boundary value problem
\begin{equation}\label{1.1}
\left\{
\begin{aligned}
&\text{div}\left({{\left|{\nabla u}\right|}^{\mathbf{p}-2}
\nabla u}\right)=0&&\text{in}\ K\cap N,\\
&u>0&&\text{in}\ K,\\
&u=0&&\text{on}\ \partial K.
\end{aligned}
\right.
\end{equation}
Here, $N$ is chosen so that the solution $u_K$ satisfies
$\left\|u_K\right\|_{L^\infty\left(\bar N\cap K\right)}
+\left\|\nabla u_K\right\|_{L^\infty\left(\bar N\cap K\right)}<\infty$
and $\left|{\nabla u_K}\right|\ne0$ in $K\cap N$,
where ${\left\|\cdot\right\|_{L^\infty}}$ is the ${L^\infty}$ norm, $\nabla$ is
the gradient operator and $\bar N$ is the closure of $N$.
Throughout this paper, we assume that $\partial N$ is of class $C^{\infty}$.
Let $W^{1,\mathbf{p}}$ denote the usual Sobolev space with $1<\mathbf{p}<\infty$.
Following Akman-Mukherjee \cite{AM24}, the $\mathbf{p}$-harmonic function
$u_K\in W^{1,\mathbf{p}}\left(K\cap N\right)$ can be used to define the measure
$\omega_\mathbf{p}
=\left|\nabla u_K\right|^{\mathbf{p}-1}
\mathcal{H}^{n-1}\llcorner_{\partial K}$.
Moreover, the $\mathbf{p}$-harmonic measure $\mu_K$ is defined by
$\mu_K=(g_K)_*\omega_\mathbf{p}$, that is,
\begin{equation}\label{1.2}
\mu_K\left(E\right)
=\int_{g_K^{-1}\left(E\right)}
{\left|\nabla u_K\right|}^{\mathbf{p}-1}
d{\mathcal H}^{n-1}
\end{equation}
for any Borel set $E$ on the unit sphere $\mathbb{S}^{n-1}$, where $g_K:\partial K\to\mathbb{S}^{n-1}$
is the Gauss map and $\mathcal{H}^{n-1}$ is the $(n-1)$-dimensional Hausdorff measure.

According to Akman-Mukherjee \cite{AM24}, the definition \eqref{1.2} is valid for any convex set,
and the $\mathbf{p}$-harmonic measure is of variation meaning. In fact, the $\mathbf{p}$-harmonic measure
has been studied by Lewis et al. \cite{L06, L13}, and Jerison's work \cite{J91} on harmonic measure has been
nontrivially extended to the $\mathbf{p}$-harmonic measure setting by Akman-Mukherjee \cite{AM24}.
By studying the discrete measure case and using the approximation arguments, Akman-Mukherjee \cite{AM24}
demonstrated the solvability of the Minkowski problem for $\mathbf{p}$-harmonic measure, provided that
the given measure is not concentrated on any great subsphere and its centroid is at the origin.
Recently, smooth solutions have been established by using the Gauss curvature flow \cite{LZ24}.
Detailed discussions on the relationships among the Minkowski problem for $\mathbf{p}$-harmonic measure,
harmonic measure \cite{J91}, and $\mathbf{p}$-capacitary measure \cite{CZ15} can be found on page 13 of \cite{AM24}.

In this paper, we focus on the following problem concerning the $\mathbf{p}$-harmonic measure,
where $1<\mathbf{p}<\infty$, unless specified otherwise.

\vskip.2cm

\textbf{$L_q$ Minkowski problem for $\mathbf{p}$-harmonic measure.}
{\it Let $q\in\mathbb{R}$ and $\mu$ be a finite Borel measure on $\mathbb{S}^{n-1}$.
What are the necessary and sufficient conditions for $\mu$ such that there exists
a convex body $\Omega$ satisfying
$\mu=h_{\Omega}^{1-q}\mu_\Omega$? Here $h_{\Omega}$ is the support function of $\Omega$.
}

\vskip.2cm

Actually, the measure $h_{\Omega}^{1-q}\mu_\Omega=\mu_{\Omega,q}$ in the above problem
can be derived from our new variational formula (see Theorem \ref{th:3.1} below),
and we call it the $L_q$ $\mathbf{p}$-harmonic measure. As mentioned above,
the $L_1$ Minkowski problem for $\mathbf{p}$-harmonic measure was recently studied
by Akman-Mukherjee \cite{AM24}. By studying an extremal problem for a functional
related to the $\mathbf{p}$-harmonic measure, we can obtain a solution to
the $L_q$ Minkowski problem for $\mathbf{p}$-harmonic measure for $0<q<1$.
This can be stated as main result of this paper as follows.

\begin{theorem}\label{th:1.1}
Let $0<q<1$, $1<\mathbf{p}\ne n+1$, and $\mu$ be a finite Borel measure on $\mathbb{S}^{n-1}$. If $\mu$ is not concentrated on any closed hemisphere,
there exists a convex body $\Omega$ containing the origin
in its interior so that $\mu=c\mu_{\Omega,q}$,
where $c$ is a positive explicit constant. In particular $c=1$,
if $\mathbf{p}\ne n+1-q$.
\end{theorem}

This paper is organized as follows. In Section \ref{sect:2}, we review some necessary notations
and background on convex sets, $\mathbf{p}$-harmonic functions and $\mathbf{p}$-harmonic measures.
In Section \ref{sect:3}, after establishing a variational formula associated with the
$\mathbf{p}$-harmonic measure, we further introduce the $L_q$ $\mathbf{p}$-harmonic
measure for $q\in\mathbb{R}$ and prove its weak convergence.
In Section \ref{sect:4}, we complete the proof of Theorem \ref{th:1.1}.

\section{Preliminaries}\label{sect:2}
\subsection{Background for convex sets}\label{subsect:2.1}

In this subsection, we collect the necessary background, notations and preliminaries.
More details on convex sets can be found in \cite{G06, G07, S14}.

Let $K\subset \mathbb{R}^{n}$ be a convex set with boundary $\partial K$,
one can define the multi-valued Gauss map $g_K:\partial K\to\mathbb{S}^{n-1}$ by

\begin{equation}\label{2.1}
{g_K}\left(x\right)
=\left\{{\xi\in {\mathbb{S}^{n-1}}:
\left\langle{y-x,\xi }\right\rangle< 0\ \text{for all}\ y\in K}\right\},
\end{equation}
i.e., the set of all unit outward normal vectors at $x\in \partial K$,
where $\left\langle{\cdot, \cdot}\right\rangle$ is the standard inner product
on $\mathbb{R}^{n}$.
The set defined in \eqref{2.1} is a singleton for
$\mathcal{H}^{n-1}$-a.e. $x\in\partial K$. For a measurable subset $E\subset\mathbb{S}^{n-1}$,
let $g_K^{-1}(E):=\{{x\in\partial K:g_K(x)\cap E\ne\emptyset}\}$ be the inverse
image of $g_K$, and ${\left(g_K\right)_*}$ be the push forward of $g_K$ given by
\[\left({{\left(g_K\right)}_*}\mu\right)\left(E\right)
=\mu\left({g_K^{-1}\left(E\right)}\right),\]
where $\mu$ is a measure defined on any measurable subsets of $\partial K$.
If $E$ is a Borel subset of $\mathbb{S}^{n-1}$, $g_K^{-1}\left(E\right)$
is $\mathcal{H}^{n-1}$-measurable.

For a compact convex set $K\subset\mathbb{R}^{n}$ and nonzero $x\in\mathbb{R}^{n}$,
the support function of $K$ is defined by
$h_K\left(x\right)
=\max\limits_{y\in K}\left\langle {x,y}\right\rangle$,
and the support hyperplane of $K$ is given by

$${H_K}(x)
=\left\{{y\in {\mathbb{R}^n}:\left\langle{x,y} \right\rangle
={h_K}(x)}\right\}.$$
If $K\cap{H_K}\left( x \right)$ consists of only a single point for all $x$, then $K$ is strictly convex.
In particular, a convex and compact subset in $\Rn$ with nonempty interior is called a convex body.

A convex set $K$ is said to be of class $C_+^2$ (resp. $C_ +^{2,\alpha }$
for $\alpha\in\left({0,1}\right]$)  if $\partial K$ is of class $C_+^2$
(resp. $C_+^{2,\alpha}$) and the Gauss map $g_K: \partial K\to\mathbb{S}^{n-1}$ is
a diffeomorphism. For any convex set $K$ of class $C_+^{2}$, we have
$K\cap {H_K}\left( {{g_K}\left( x \right)} \right)
=\left\{ x \right\}$, where $x\in \partial K$.
Moreover, the support function
is differentiable and
\[\nabla {h_K}\left( {{g_K}\left( x \right)} \right) = x,\]
where $\nabla $ is the gradient operator on $\mathbb{R}^{n}$.
For $\xi \in \mathbb{S}^{n-1}$, there exists an orthonormal basis
$\left\{ {{e^1}, \ldots ,{e^{n - 1}},\xi } \right\}$ of $\mathbb{R}^{n}$, where $\left\{ {e^i}\right\}$
spans the tangent space ${T_\xi }\left(\mathbb{S}^{n-1}\right)$. Then, for any $x\in \mathbb{R}^{n}$,
we have the decomposition
\begin{equation}\nonumber
x=\sum\limits_{i=1}^{n-1}x^ie^i
+\left\langle{x,\xi}\right\rangle\xi\ \ \text{with}\ \
x^i=\left\langle x,e^i\right\rangle.
\end{equation}
Let $\xi ={g_K}\left( x \right)$ for any $x\in\partial K$, then we have
\begin{equation}\label{2.2}
\nabla {h_K}\left(\xi\right)
=\sum\limits_{i = 1}^{n - 1} {{\bigtriangledown_i}{h_K}\left( \xi  \right){e^i}}
+\left\langle {\nabla {h_K}\left( \xi  \right),\xi } \right\rangle \xi,
\end{equation}
where
${\bigtriangledown _i}{h_K}\left(\xi\right)
=\left\langle {\nabla {h_K}\left(\xi\right),{e^i}}\right\rangle$.

Let $\mathcal{A}_+^{2,\alpha}$ be the set of all compact convex sets that are of class $C_+^{2,\alpha}$.
For a sequence of compact convex sets $\left\{\Omega_j\right\}_{j=0}^{\infty}$, we say that $\Omega_j$
converges to $\Omega_0$ and denote it as $\Omega_j\to \Omega_{0}$, if the Hausdorff distance $d_{\mathcal H}\left({\partial \Omega_j,\partial\Omega_0}\right)$ between
${\Omega_j}$ and $\Omega_{0}$ converges to $0$ as $j \to\infty$.
According to Theorem 2.46 of \cite{AM24}, for any compact convex set $\Omega$ with Gaussian
curvature $\kappa$, there exists a sequence
$\left\{\Omega_j \right\}_{j=1}^\infty\subset\mathcal{A}_+^{2,\alpha}$ with Gaussian curvature
$\kappa_{j}$ such that $\Omega_{j}\to \Omega$, and for any continuous function $f$ defined on the unit sphere $\mathbb{S}^{n-1}$,
\begin{equation}\nonumber
\int_{\mathbb{S}^{n-1}}\frac{f\left(\xi\right)}{\kappa_j\left({g_{\Omega_j}^{-1}\left(\xi\right)}\right)}d\xi
\to
\int_{\mathbb{S}^{n-1}}\frac{f\left(\xi\right)}
{\kappa\left({g_\Omega^{-1}\left(\xi\right)}\right)}d\xi,
\end{equation}
as $j\to\infty$.

Let $C\left(E\right)$ denote the set of all continuous functions
defined on subset $E\subset\mathbb{S}^{n-1}$ and let $C_{+}\left(E\right)\subset C\left(E\right)$ denote
the set of all strictly positive functions. The Wulff shape
$K_f$ associated with a nonnegative function $f\in C\left(E\right)$
is defined by
\begin{equation}\nonumber
{K_f}
=\left\{{x\in\mathbb{R}^{n}:\left\langle {x,\xi}\right\rangle
\le f\left(u\right)}\
\text{for all}\ \xi\in E\right\}.
\end{equation}
Let $\mathcal K_o^n$ be the set of convex bodies containing
the origin $o$ in their interiors.
A well-known fact is that $K_f\in\mathcal K_o^n$ if $f\in C_{+}\left(\mathbb{S}^{n-1}\right)$, and $h_{K_f}=f$
almost everywhere with respect to the surface area measure of $K_f$.
Schneider \cite{S14} proved that if $\{f_j\}_{j=1}^\infty\subset C_{+}\left(\mathbb{S}^{n-1}\right)$
converges to $f\in C_{+}\left(\mathbb{S}^{n-1}\right)$ uniformly as $j\to\infty$, then the sequence
$\{K_{f_j}\}$ is also convergent in the sense of the Hausdorff metric, i.e.,
\begin{equation}\label{2.3}
K_{f_j}\to K_f,\ \text{as}\ j\to\infty.
\end{equation}

\subsection{The $\mathbf{p}$-harmonic functions and
$\mathbf{p}$-harmonic measures}\label{subsect:2.2}

We now review some properties of the $\mathbf{p}$-harmonic function, which are also
referenced in \cite{AM24} for more details.

The $\mathbf{p}$-harmonic functions minimize the $\mathbf{p}$-Dirichlet energy
$\int_{K}{\left|{\nabla u}\right|}^\mathbf{p}dx$
and are weak solutions to the $\mathbf{p}$-Laplacian equation
$\Delta_\mathbf{p}u=\text{div}\left({{\left|{\nabla u}\right|}^{\mathbf{p}-2}
\nabla u}\right)=0$ in a convex domain $K$.
The existence of a weak solution $u_K\in W^{1,\mathbf{p}}\left(K\right)$ to $\Delta_\mathbf{p}u=0$
in $K$, with boundary condition $u=f$ on $\partial K$, is known. The uniqueness of the weak solution
follows directly from the comparison principle, while the regularity theory presents more complex challenges.
Let $K\in \mathcal{A}_+^{2,\alpha}$ and $f\in C^{1,\alpha}\left(\partial K\right)$,
it follows from \cite{L88} that
$u_K\in C^{1,\beta}\left(\bar{K}\right)$ for some $\beta(n,\mathbf{p},\alpha)\in(0,1)$.
Tolksdorf \cite{T84} has proved that the weak solutions to $\Delta_\mathbf{p}u = 0$ in $K$
are locally $C^{1,\beta}$ for some $\beta(n,\mathbf{p})\in \left(0,1\right)$.
This shows that for any compact subset $K^\prime\subset \subset K$,
the weak solutions are continuously differentiable on $K^\prime$ and their first derivatives are H\"older continuous.
Hence, the weak solution $u$ to \eqref{1.1} belongs to $C^{1,\beta}(\bar K\cap N)$.
Since $\left|{\nabla u}\right|\ne0$ in $K\cap N$, the $\mathbf{p}$-Laplacian operator
is uniformly elliptic in $K\cap N$. It follows from the boundary Schauder estimates \cite{GT01}
that the Hessian matrix $D^{2}u$ is well-defined on $\partial K$.
Let $u_{K_j}$ be the weak solution to \eqref{1.1} for $K_j$. Then, by Proposition 3.65 of \cite{AM24}, $\nabla u_{K_j}\to\nabla u_K$ uniformly in $N$, if $K_j\to K$.

For the $\mathbf{p}$-harmonic function, we provide two important lemmas. The first one can be stated as follows.

\begin{lemma}\label{lem:2.1}
Let $K$ be a bounded convex domain containing the origin and $u$ be the solution to \eqref{1.1},
there exists a constant $M>0$, independent of $K$, such that
\begin{equation}\nonumber
\left|{\nabla u}\right|\le M\ \mathrm{on}\ \partial K.
\end{equation}
\end{lemma}

\begin{proof}
By Theorem 2.46 of \cite{AM24}, for any convex domain $K$, there exists a  sequence of convex domains $\{K_j\}\subset\mathcal{A}_+^{2,\alpha}$
that converges to $K$ as $j\to\infty$. Thus, we only need to consider the case that $K\in \mathcal{A}_+^{2,\alpha}$.

Let $u$ be a solution to the boundary value problem
\begin{equation}\label{2.4}
\left\{
\begin{aligned}
&\text{div}\left({{\left|{\nabla u}\right|}^{\mathbf{p}-2}
\nabla u}\right)=0&&\text{in}\ K\setminus \bar\Omega_0,\\
&u>0&&\text{in}\ K,\\
&u=0&&\text{on}\ \partial K,\\
\end{aligned}
\right.
\end{equation}
where $\bar\Omega_0:=K\setminus N$. If $u=1$ in $\bar\Omega_0$,
it follows from page 204 of \cite{L77} that $u$ is a $\mathbf{p}$-capacity function of
$K\setminus\bar\Omega_0$. By Theorem 2 of \cite{CS03}, we conclude that
$u\in C^{\infty}\left(K \setminus \bar\Omega_0\right)
\cap C\left(K \setminus \Omega_0\right)$, $0<u<1$ in $K \setminus\bar\Omega_0$
and $K_s= \left\{ {x\in K:u(x)\ge s} \right\}$ is convex for $0\le s\le 1$.

Since $\left| {\nabla u\left( x \right)} \right| > 0$ in $K \setminus \bar \Omega_0$,
by Theorem 4 of \cite{CS03}, we obtain
\begin{equation}\label{2.5}
- \frac{{\partial {h_{{K_s}}}\left( {{{-\nabla u\left( x \right)} \mathord{\left/
{\vphantom {{\nabla u\left( x \right)} {\left| {\nabla u\left( x \right)} \right|}}} \right.
\kern-\nulldelimiterspace} {\left| {\nabla u\left( x \right)} \right|}}} \right)}}{{\partial s}}
=\frac{1}{{\left| {\nabla u\left( x \right)} \right|}},
\end{equation}
for all $x\in\partial K_s$. By applying Proposition 1 of \cite{CS03}, we further have
\[\frac{{\partial {h^2_{{K_s}}}\left( {{{-\nabla u\left( x \right)} \mathord{\left/
{\vphantom {{\nabla u\left( x \right)} {\left| {\nabla u\left( x \right)} \right|}}} \right.
\kern-\nulldelimiterspace} {\left| {\nabla u\left( x \right)} \right|}}} \right)}}{{\partial s^{2}}} \ge 0,\]
thus $\frac{{\partial {h_{{K_s}}}\left( {{{-\nabla u\left( x \right)} \mathord{\left/
{\vphantom {{\nabla u\left( x \right)} {\left| {\nabla u\left( x \right)} \right|}}} \right.
\kern-\nulldelimiterspace} {\left| {\nabla u\left( x \right)} \right|}}} \right)}}{{\partial s}} $
is non-decreasing for every fixed $x$. This, together with \eqref{2.5}, shows that
$\left|{\nabla u\left(x\right)}\right|$
attains its maximum on $\partial\bar\Omega_0$.
Let $B_r$ be a ball with radius $r$ included in $\bar\Omega_0$ and
internally tangent to $\partial\bar\Omega_0$ at $x\in\partial\bar\Omega_0$,
and let $v$ be a solution to the equation \eqref{2.4} with $\bar \Omega_0$ replaced by $B_r$.
As $B_{r}\subset \bar \Omega_0$, we have $K \setminus \bar \Omega_0 \subset K \setminus B_{r}$,
thus
\begin{equation*}
\left\{
\begin{aligned}
&\Delta_{\mathbf{p}}u= \Delta _{\mathbf{p}}v && \text{in}\ K \setminus\bar\Omega_0,\\
&u=v=0\ &&\text{on}\ \partial K,\\
&v\le u\ &&\text{on}\ \partial \Omega_0.\\
\end{aligned}
\right.
\end{equation*}
Then, by the comparison principle (cf. Theorem 2.1 of \cite{G13}), $v\le u$ on
$K \setminus\bar\Omega_0$. This, combined with $u(x)=v(x)$, implies that
$\left|{\nabla u\left(x\right)}\right|\le\left|{\nabla v\left(x\right)} \right|$
for $x\in\partial\bar\Omega_0$.
Then, we can calculate the value of $\left|{\nabla v\left(x\right)}\right|$ and obtain a positive constant $m$
depending on $r$ and $n$ such that
\begin{equation}\label{2.6}
\left|{\nabla u}\right|\le m
\end{equation}
in $K\setminus\bar\Omega_0$.

Moreover, since $u\in C^{1,\beta}\left({\bar K\cap N}\right)$
with $\beta=\beta(n,\mathbf{p},\alpha)$, it follows that $\nabla u$ is $\beta$-H\"older continuous.
Then, there exists a constant $\Lambda>0$ such that
$$
\left| {\nabla u\left( y \right)-\nabla u\left(z\right)}\right|\le \Lambda{\left|{y-z} \right|^\beta}
$$
for $y,z\in {\bar K\cap N}$. Thus, we have
\[\left| {\nabla u\left(z \right)} \right| \le \Lambda {\left| {y-z} \right|^\beta } + \left| {\nabla u\left(y\right)} \right|\]
for any $z\in \partial K$ and $y\in K\cap N$.
This, together with \eqref{2.6} and the boundedness of ${\bar K\cap N}$, shows that
there exists a finite positive constant $M$, independent of $K$, such that
$$|{\nabla u\left(z\right)}|\le M$$ for all $z\in\partial K$.
This completes the proof of Lemma \ref{lem:2.1}.
\end{proof}

The second order covariant derivative of $h_K:\mathbb{S}^{n-1}\to\mathbb{R}$ is locally given by
$${\bigtriangledown}^2{h_K}
=\sum\limits_{{i,j= 1}}^{n-1}(\bigtriangledown_{i,j}h_K)
e^i\otimes e^j,$$
where $\bigtriangledown_{i,j}h_K(x)=\partial_{i,j}(h_K\circ\varphi^{-1})(\varphi(x))$ with $U\subset\mathbb{S}^{n-1}$
and $\varphi:U\to V\subset\mathbb{R}^{n-1}$ being a coordinate chart.
Let $\mathbb{I}$ be the unit matrix of order $(n-1)$ and $C[\bigtriangledown^2h_K+h_K\mathbb{I}]$ be the cofactor matrix of
$\left({{\bigtriangledown^2}{h_K}+{h_K}{\mathbb{I}}}\right)$ with element
${C_{i,j}}\left[\cdot\right]
=\left\langle{C\left[\cdot\right]{e^j},{e^i}}\right\rangle$.
The following lemma directly follows from Lemma 3.44 of \cite{AM24}.

\begin{lemma}\label{lem:2.2}
Let $\left\{{e^1,\ldots,e^{n-1},\xi}\right\}$ be an orthonormal basis of $\mathbb{R}^{n}$,
and let $u$ be the solution to \eqref{1.1} for a convex domain $K$ that is of class
$C_+^{2,\alpha}$. Then we have
\begin{enumerate}[label=\upshape(\roman*)]
\item 	$\left\langle {{D^2}u\left( {\nabla {h_K}\left( \xi  \right)} \right)\xi ,\xi } \right\rangle
= \frac{1}{{\mathbf{p}-1}}\kappa \left( {\nabla {h_K}\left( \xi  \right)} \right)\left| {\nabla u\left( {\nabla {h_K}
\left( \xi  \right)} \right)} \right|{\rm{Tr}}\left( {C\left[ {{\bigtriangledown ^2}{h_K} + {h_K}{\mathbb{I}}} \right]} \right)$,
\item $\left\langle {{D^2}u\left( {\nabla {h_K}\left( \xi  \right)} \right){e^i},\xi } \right\rangle
=-\kappa \left( {\nabla {h_K}\left( \xi  \right)} \right)
\sum\limits_{j = 1}^{n - 1} {{C_{i,j}}\left[ {{\bigtriangledown ^2}{h_K} + {h_K}{\mathbb{I}}} \right]} {\bigtriangledown _j}
\left( {\left| {\nabla u\left( {\nabla {h_K}\left( \xi  \right)} \right)} \right|} \right)$.
\end{enumerate}
\end{lemma}

At the end of this subsection, we review the weak convergence of the  $\mathbf{p}$-harmonic measure.
Let $u\in W^{1,\mathbf{p}}\left(K\cap N\right)$ be a $\mathbf{p}$-harmonic function, a solution to \eqref{1.1} in $K\cap N$.
Following Akman-Mukherjee \cite{AM24}, one can define
the $\mathbf{p}$-harmonic measure
\begin{equation}\nonumber
{\mu_{\Bar K}}\left(E\right)
={\mu_K}\left(E\right)
=\int_{g_K^{-1}\left(E\right)}{\left|
{\nabla u\left(x\right)} \right|}^{\mathbf{p}-1}d{\mathcal{H}}^{n-1}\left(x\right),
\end{equation}
where $E\subset\mathbb{S}^{n-1}$ is a Borel subset.
If $K\in \mathcal{A}_+^{2,\alpha }$, we have $\nabla h_K\left(\xi\right)=g_K^{-1}\left(\xi\right)$,
and we can use the transformation rule of
the Jacobian (cf. page 8 of \cite{AM24}) to obtain
\begin{equation}\label{2.7}
(g_K)_*\mathcal{H}^{n- 1}\llcorner_{\partial K}
=|\det\left({\bigtriangledown}^2h_K+h_K{\mathbb{I}}\right)|
\mathcal{H}^{n- 1}\llcorner_{\mathbb{S}^{n-1}}
=\frac{1}{\left(\kappa\circ g_K^{-1}\right)}
\mathcal{H}^{n- 1}\llcorner_{\mathbb{S}^{n-1}}.
\end{equation}
Therefore,
\begin{equation}\nonumber
\begin{split}
d{\mu _K} = {\left| {\nabla u\left( {\nabla {h_K}\left(\xi\right)} \right)} \right|^{\mathbf{p}-1}}d{\mathcal{H}^{n - 1}} \llcorner_{\partial K}={\left| {\nabla u\left( \nabla {h_{{K}}}\left(\xi\right) \right)}
\right|^{\mathbf{p}-1}}\det \left( {{\bigtriangledown ^2}{h_K} + {h_K}\mathbb{I}} \right)d\xi.
\end{split}
\end{equation}
For a compact convex set $K$ and a sequence of compact convex sets $\left\{K_{j}\right\}$ with $K_{j}\to K$ as $j\to\infty$,
Akman-Mukherjee \cite{AM24} proved that
\begin{equation}\label{2.8}
\mathop{\lim}\limits_{j\to\infty }
\int_{\mathbb{S}^{n-1}}{f\left(\xi\right)}d\mu_{K_j}\left(\xi\right)
=\int_{\mathbb{S}^{n-1}}{f\left(\xi\right)} d{\mu_K}\left(\xi\right)
\end{equation}
for any $f \in C\left(\mathbb{S}^{n-1}\right)$. This shows that the $\mathbf{p}$-harmonic measure is weakly convergent.
Moreover, it can be checked that the centroid of the $\mathbf{p}$-harmonic measure is at the origin.

\begin{lemma}\label{lem:2.3}
Let $K$ be a bounded convex domain, then for any $x_0\in \mathbb{R}^{n}$,
$$
\int_{{\mathbb{S}^{n - 1}}} {\left\langle {{x_0},\xi } \right\rangle } d{\mu_K}(\xi)=0.
$$
\end{lemma}

\begin{proof}
Let $u_K$ be a weak solution to the $\mathbf{p}$-Laplace equation
in $K \cap N$, or equivalently,
\begin{equation}\label{2.9}
\int_{K\cap N} {{{\left|\nabla u_K(x)\right|}^{\mathbf{p}-2}}
\left\langle{\nabla u_K(x),\nabla\phi(x)}\right\rangle}dx
=0
\end{equation}
for any smooth function $\phi$ defined in $K\cap N$ with compact support. Consider the boundary value problem \eqref{1.1} and
let $f$ be a function in $C^\infty\left(\overline{K\cap N}\right)$ such that $f=u_K$ on $\partial N\cap K$ and $f=1$ on $\partial K$.
Notice that $$g_K(x)=-\frac{\nabla u_K(x)}{\left|{\nabla u_K(x)}\right|},$$
then for any $x_0\in \mathbb{R}^{n}$, we have the following calculation:
\begin{equation*}
\begin{split}
&\int_{{\mathbb{S}^{n - 1}}} {\left\langle {{x_0},\xi } \right\rangle } d{\mu_K}\left( \xi  \right)\\
=&\int_{{\mathbb{S}^{n - 1}}} {\left\langle {{x_0},\xi } \right\rangle } {\left| {\nabla {u_K}
\left( {g_K^{ - 1}\left( \xi  \right)} \right)} \right|^{{\mathbf{p}} - 1}}d{S_K}\left( \xi  \right)\\
=&\int_{\partial K} {{{\left| {\nabla {u_K}\left( x \right)} \right|}^{{\mathbf{p}} - 1}}
\left\langle {{x_0}, g_K(x)} \right\rangle } d{\mathcal{H}^{n - 1}}\\
=&\int_{\partial K} {{{\left| {\nabla {u_K}\left( x \right)} \right|}^{{\mathbf{p}} - 2}}} \left\langle {\nabla {u_K}\left( x \right),{g_K}\left( x \right)} \right\rangle
\left\langle {{x_0}, g_K(x)} \right\rangle  \left( {u_K\left( x \right)-f\left( x \right) } \right)d{\mathcal{H}^{n - 1}}\\
&+\int_{\partial N \cap K} {{{\left| {\nabla {u_K}\left( x \right)} \right|}^{{\mathbf{p}} - 2}}} \left\langle {\nabla {u_K}\left( x \right),{\nu _{\partial N \cap K}}\left( x \right)} \right\rangle \left\langle {{x_0}, g_K(x)} \right\rangle  \left( {u_K\left( x \right)-f\left( x \right) } \right)d{\mathcal{H}^{n - 1}}\\	
=&\int_{\partial \left( {K \cap N} \right)} {{{\left| {\nabla {u_K}\left( x \right)} \right|}^{{\mathbf{p}} - 2}}\left\langle {\nabla {u_K}\left( x \right),{\nu _{\partial \left( {K \cap N} \right)}}\left( x \right)} \right\rangle \left\langle {{x_0}, g_K(x)} \right\rangle \left( {u_K\left( x \right) - f\left( x \right)} \right)} d{\mathcal{H}^{n - 1}}\\
=&\int_{K \cap N} {\text{div}\left( {{{\left| {\nabla {u_K}\left( x \right)} \right|}^{{\mathbf{p}} - 2}}\nabla {u_K}\left( x \right)\left\langle {{x_0}, g_K(x)} \right\rangle \left( {u_K\left( x \right) - f\left( x \right)} \right)} \right)} dx\\
=& 0,\\ 	
\end{split}
\end{equation*}
where we have used the divergence theorem and \eqref{2.9}.
This proves the desired property.
\end{proof}

\section{The variational formula associated with $\mathbf{p}$-harmonic measure}\label{sect:3}

Associated with the $\mathbf{p}$-harmonic measure $\mu_K$ of
a compact convex set $K\subset\R^n$, Akman-Mukherjee \cite{AM24} introduced a continuous functional
\begin{equation}\label{3.1}
\Gamma\left(K\right)
=\int_{\mathbb{S}^{n-1}}h_K\left(\xi\right) d{\mu_K}\left(\xi\right).
\end{equation}
By Lemma \ref{lem:2.3}, it can be verified that the functional $\Gamma(\cdot)$ is translation invariant. That is, for any $x_0\in \mathbb{R}^{n}$,
\begin{equation}\label{3.2}
\Gamma \left({K+x_0}\right)=\Gamma\left(K\right).
\end{equation}

In the following part of this section, we will focus on calculating the variation of $\Gamma\left(K\right)$ with respect to the $q$-sum for $q>0$
and introduce the $L_q$ $\mathbf{p}$-harmonic measure. To do so, we will briefly review the concept of the $q$-sum.

Let $K$ and $L$ be two compact convex sets containing the origin.
For $q\ge1$ and $t\ge 0$, Firey's $q$-sum $K^t$ can be defined by
$h_{K^t}^q=h_K^q+th_L^q$ on $\mathbb{S}^{n-1}$.
Following B\"or\"oczky  et al. \cite{BZ12}, the $q$-sum $K^t$ for $0<q<1$ can be
defined as the Wulff shape of the function
$\left(h_K^q+ th_L^q\right)^{\frac{1}{q}}$, that is
\begin{equation}\label{3.3}
{K^t}=\left\{{x\in \mathbb{R}^{n}:\left\langle {x,\xi}\right\rangle
\le{\left( {h_K^q\left(\xi\right)
+th_L^q\left(\xi\right)} \right)}^{\frac{1}{q}}}\
\text{for all}\ \xi\in\mathbb{S}^{n-1}\right\}.
\end{equation}
In this case, $h_{K^t}^q=h_K^q+th_L^q$ holds almost everywhere on
$\mathbb{S}^{n-1}$ with respect to the surface area measure $S_{K^t}$ of $K^t$. Thus, we have
$S_{K^t}\left(\omega_t\right)=0$,
where
\[{\omega_t}=\left\{\xi\in {\mathbb{S}^{n- 1}}:h_{{K^t}}^q(\xi)\ne h_K^q
(\xi)+th_L^q(\xi)\right\}.\]

Let $K,L\in \mathcal{A}_+^{2,\alpha}$ and $q>0$. We take a small enough
\begin{equation}\label{3.4}
\tau
:=\tau\left(d_{\mathcal H}\left({\partial K,\partial N}\right),
d_{\mathcal H}\left({\partial L,\partial N}\right),
\left\|u\right\|_{W^{1,\mathbf{p}}\left(N\right)}\right)
>0,
\end{equation}
where $u$ is the solution to \eqref{1.1},
such that ${K^t}\in\mathcal{A}_+^{2,\alpha}$, $\partial K^{t}\subset N$, and
$K^{t}\cap\partial N=K\cap\partial N$ for all $\left|t\right|\le\tau$.
With this choice, we conclude that
$g_{K^t}:\partial K^t\to\mathbb{S}^{n-1}$ is a diffeomorphism. It follows that ${\mathcal{H}^{n-1}}\left({\omega _t}\right)=0$ and
\[\int_{\mathbb{S}^{n-1}}h_{K^t}^qd\xi
=\int_{\mathbb{S}^{n-1}}{(h_K^q+th_L^q)}d\xi.\]

Next, we consider the $\mathbf{p}$-harmonic measure corresponding
to $u(\cdot,t)\in W^{1,\mathbf{p}}(K^{t}\cap N)$, which is a weak solution
to the Dirichlet problem
\begin{equation}\label{3.5}
\left\{
\begin{aligned}
&\text{div}\left({{\left|{\nabla u\left(x,t\right)}\right|}^{\mathbf{p}-2}
\nabla u\left(x,t\right)}\right)=0&&x\in K^t\cap N,\\
&u\left(x,t\right)=0&&x\in\partial K^t,\\
&u\left(x,t\right) = u\left(\frac{x}{\left(1+t\right)^{\frac{1}{q}}}\right)&&x\in\partial N\cap K^t,
\end{aligned}
\right.
\end{equation}
where $\left| t \right|$ is small enough so that upon zero extension,
$u\left(x,t\right) \in {W^{1,\mathbf{p}}}\left( N \right)$.
By defining
\begin{equation}\label{3.6}
\mathcal{F}\left[h_{K^t}\right]\left(\xi\right)
:={\left|{\nabla u\left({\nabla{h_{K^t}}
\left(\xi\right),t}\right)}\right|^{\mathbf{p}-1}}\det
\left({{\bigtriangledown^2}{h_{{K^t}}}+h_{K^t}\mathbb I}\right),
\end{equation}
we obtain
\begin{equation*}
d{\mu_{{K^t}}}
={\left|{\nabla u\left({\nabla{h_{{K^t}}}\left(\xi\right),t} \right)}\right|^{\mathbf{p}-1}}
d{\cal H}^{n-1}{\llcorner_{\partial K^t}}
=\mathcal{F}\left[h_{K^t}\right]\left(\xi\right)d\xi,
\end{equation*}
and
\begin{equation}\label{3.7}
\Gamma\left(K^t\right)
=\int_{\mathbb{S}^{n-1}}{{h_{K^t}}\left(\xi\right)}
d{\mu_{K^t}}\left(\xi\right)
=\int_{{\mathbb{S}^{n-1}}}h_{K^t}
\left(\xi\right)\mathcal{F}\left[h_{K^t}\right]
\left(\xi\right)d\xi.
\end{equation}

\begin{lemma}\label{lem:3.1}
Let $1<\mathbf{p}<\infty$ and $q>0$, and let $\mathcal{F}$ be given by \eqref{3.6}. Then we have
\begin{equation}\label{3.8}
\mathcal{F}\left[ {\left( {1 + t} \right)^{\frac{1}{q}}{h_K}} \right]\left( \xi  \right)
= {\left( {1 + t} \right)^{\frac{n-\mathbf{p}}{q}}}\mathcal{F}\left[ {{h_K}} \right]\left( \xi  \right),
\end{equation}
for all $\left| t \right|\le\tau$.  Here $\tau$ is given in \eqref{3.4}.
\end{lemma}

\begin{proof}
The proof is similar to that of Lemma 3.12 in \cite{AM24}. For completeness, we provide a proof as follows.

We first deal with the case that $0<q<1$. By setting $L=K$ in \eqref{3.3},
we obtain that $K^{t}=\lambda K$ is the Wulff shape
of the support function $\lambda h_{K}$, where $\lambda=\left(1+t\right)^{\frac{1}{q}}$.
Let $u_{\lambda}\left(\cdot\right):=u\left(\cdot,\lambda^{q}-1\right)$ be the weak solution
to the Dirichlet problem
\begin{equation}\label{3.9}
\left\{
\begin{aligned}
&\text{div}\left( {{{\left| {\nabla u_{\lambda}\left( x \right)} \right|}^{\mathbf{p}-2}}
\nabla u_{\lambda}\left( x \right)} \right) = 0&&x\in\lambda K\cap N,\\
&u_{\lambda}\left(x\right)=0&&x\in\partial(\lambda K),\\
&u_{\lambda}\left(x\right)= u\left(\frac{x}{\lambda}\right)&&x\in\partial N\cap\lambda K,
\end{aligned}
\right.
\end{equation}
for $\left| {{\lambda ^q} - 1} \right| \le \tau $.
Then we have
\begin{equation}\label{3.10}
\begin{split}
\mathcal{F}\left[ {{\lambda h_{K}}} \right]\left( \xi \right)
&={\left| {\nabla u_{\lambda}\left( {\lambda \nabla {h_{K}}\left( \xi  \right)} \right)}
\right|^{\mathbf{p}-1}}\lambda^{n-1}\det \left( {{\bigtriangledown ^2}{h_{K}}+{h_{K}}\mathbb{I}} \right)\\
&= \left({\frac{\left|{\nabla {u_\lambda }\left( {\lambda \nabla {h_K}\left( \xi  \right)} \right)}\right|}{\left| {\nabla u\left( {\nabla {h_K}\left( \xi  \right)} \right)} \right|}}\right)^{\mathbf{p} - 1}
{\lambda ^{n - 1}}\mathcal{F} \left[ {{ h_{K}}} \right]\left( \xi \right).
\end{split}
\end{equation}
As $u$ is the solution to \eqref{1.1}, we have that $u\left(\frac{x}{\lambda}\right)$
is also the solution to \eqref{3.9} in $\lambda K$. By the uniqueness of
the solution to \eqref{3.9},
$u_{\lambda}\left(x\right)=u\left(\frac{x}{\lambda}\right)$
in $\lambda K$. It follows that
$\nabla {u_\lambda }\left( x \right)
= \frac{1}{\lambda }\nabla u\left( {\frac{x}{\lambda }} \right)$, thus \eqref{3.10} gives
\[\mathcal{F}\left[ {\lambda {h_K}} \right]\left( \xi  \right)
= {\lambda^{n-\mathbf{p}}}\mathcal{F}\left[ {{h_K}} \right]\left( \xi  \right)\]
for $\left| {{\lambda ^q} - 1} \right| \le \tau $.
This proves the case $0<q<1$.

Note that the $q$-sum $K^t$ for $q\ge1$ can also be given by \eqref{3.3},
and the argument for the case $q\ge1$ follows along the same lines. Therefore, the remaining case of the proof is omitted.
\end{proof}

We define $\dot u\left(x\right)
={{{\left. {\frac{\partial}{{\partial t}}}\right|}_{t= 0}}u\left( {x,t} \right)}$
and present a differentiability lemma as follows.

\begin{lemma}\label{lem:3.2}
Let $1<\mathbf{p}<\infty$ and $q>0$, and let $K, L\in \mathcal{A}_+^{2,\alpha}$
be two compact convex sets containing the origin. If $u\left(\cdot,t\right)\in W^{1,\mathbf{p}}\left(K^{t} \cap N\right)$ is the solution to \eqref{3.5},
the following holds:
\begin{enumerate}[label=\upshape (\roman*)]

\item The map $t\mapsto u\left({x,t} \right)$ is differentiable at $t=0$
for all $x\in\bar K\cap N$, and $\dot u\in C^{2,\beta}\left(\overline{K\cap N}\right)$ with $\beta=\beta(n,\mathbf{p},\alpha)$;
				
\item For $x\in\partial K$ and $q\ge1$,
$\dot u(x)=\left| {\nabla u\left(x\right)}\right|
\left({\frac{1}{q}h_K^{1-q}\left(g_K\left(x\right)\right)h_L^q
\left( {{g_K}\left(x\right)}\right)}\right)$. If $0<q<1$, this equality
holds almost everywhere with respect to $S_K$.
\end{enumerate}
\end{lemma}

\begin{proof}
Part (i) comes from Proposition 3.20 of \cite{AM24}. Here, we provide a brief proof of (ii) for the case $0<q<1$; the case $q\ge1$ follows similarly.

Define $\omega\left(x,t\right)=\frac{u\left(x,t\right)-u\left(x,0\right)}{t}$
for $t\neq0$. According to (3.23) in \cite{AM24}, there exists a sequence $\{t_k\}$ such that
$t_k\to 0$ as $k\to\infty$, and the limit
\begin{equation*}
\lim\limits_{k\to\infty }\omega\left(x,{t_k}\right)
=\lim\limits_{k\to\infty }
\frac{u\left({x,{t_k}} \right)-u\left({x,0}\right)}{t_k}
=:\omega\left(x\right)
\end{equation*}
exists for all $x\in K\cap N$. Moreover, for $x\in \partial K$, there
exists a sequence $\left\{x_j\right\}\subset \text{int}K$ such that $x_j\to x$
as $j\to\infty$, and
\begin{equation*}
\omega\left(x\right)
=\lim\limits_{j\to\infty }\omega\left(x_j\right)
=\lim\limits_{j\to\infty }\lim \limits_{k\to\infty}\omega\left({x_j},{t_k}\right)
=\lim\limits_{k\to\infty }\frac{u\left(x,t_k\right)-u\left(x,0\right)}{t_k},
\end{equation*}
for any $x\in\partial K$. Hence, the function $t\to u\left({\cdot,t}\right)$
is differentiable at $t=0$ for all $x\in\bar K\cap N$.
It follows from (3.26) and (3.27) of \cite{AM24} that
$\dot u\in C^{2,\beta} \left(\overline{K\cap N}\right)$,
and
\[\left| {\omega \left( {{x_k},{t_k}} \right)
-\omega\left( {{x_k},0}\right)}\right|\leqslant\Lambda\left|{x_k-x}\right|\]
for $\Lambda>0$ and any $x_k\in\partial K^{t_k}$. Thus,
\begin{equation*}
\omega\left(x\right)
=\lim\limits_{k\to\infty}\omega\left(x_k,t_k\right)
=\lim\limits_{k\to\infty}\frac{u\left(x_k,t_k\right)-u\left(x_k,0\right)}{t_k}
=\lim\limits_{k\to\infty}\frac{u\left(x\right)-u\left(x_k,0\right)}{t_k}
\end{equation*}
for any $x\in\partial K$.

For $\xi\in\mathbb{S}^{n-1}$, there exists $x\in\partial K$ and
$x_k\in\partial K^{t_k}$ so that $x=\nabla h_{K}\left(\xi\right)$,
$x_{k}=\nabla h_{K^{t_k}}\left(\xi\right)$.
Then, we compute:
\begin{equation*}
\begin{split}
\nabla h_{K^{t_k}}
&=\nabla {\left({h_K^q+t_kh_L^q}\right)^{\frac{1}{q}}}\\
&= {\left({h_K^q+t_kh_L^q}\right)^{\frac{{1-q}}{q}}}h_K^{q - 1}\nabla {h_K}
+t_k{\left( {h_K^q + t_kh_L^q} \right)^{\frac{{1 - q}}{q}}}h_L^{q - 1}\nabla {h_L}\\
&={\left({1+t_kh_L^qh_K^{-q}} \right)^{\frac{{1 - q}}{q}}}\nabla {h_K}
+t_k{\left( {{{\left( {h_L^qh_K^{-q}}\right)}^{-1}}+t_k} \right)^{\frac{{1 - q}}{q}}}\nabla {h_L}\\
&=\nabla {h_K}+\left( {{{\left( {1 +t_kh_L^qh_K^{ - q}} \right)}^{\frac{{1 - q}}{q}}} - 1} \right)\nabla {h_K}
+ t_kh_L^{q - 1}h_K^{1 - q}{\left( {1 + t_kh_L^qh_K^{ - q}} \right)^{\frac{{1 - q}}{q}}}\nabla {h_L},
\end{split}
\end{equation*}
$S_{K^{t_k}}$-almost everywhere. Taking the limit as $k\to \infty$, we obtain:
\begin{equation*}
\begin{split}
\mathop {\lim }\limits_{k \to \infty} \frac{{{x_k} - x}}{t_k}
&= \mathop {\lim }\limits_{k \to \infty} \frac{{\left( {{{\left( {1 + t_kh_L^qh_K^{ - q}} \right)}^{\frac{{1 - q}}{q}}} - 1} \right)\nabla {h_K} + t_kh_L^{q - 1}h_K^{1 - q}{{\left( {1 +t_kh_L^qh_K^{ - q}} \right)}^{\frac{{1 - q}}{q}}}\nabla {h_L}}}{t_k}\\
&= \frac{{1 - q}}{q}h_L^qh_K^{ - q}\nabla {h_K} + h_L^{q - 1}h_K^{1 - q}\nabla {h_L}\\
&= \nabla \left( {\frac{1}{q}h_K^{1 - q}h_L^q} \right),
\end{split}.
\end{equation*}
$S_{K}$-almost everywhere. Thus,
\begin{equation*}
\begin{split}
\omega \left( x \right)
=\mathop {\lim }\limits_{k \to \infty } \frac{{u\left( x \right) - u\left( {{x_k},0} \right)}}{{{t_k}}}
=-\left\langle {\nabla u\left( x \right),\nabla \left( {\frac{1}{q}h_K^{1 - q}h_L^q} \right)} \right\rangle,
\end{split}
\end{equation*}
$S_{K}$-almost everywhere for all $x\in \partial K$.
Notice that
$\xi=-\frac{{\nabla u\left( x \right)}}{{\left| {\nabla u\left( x \right)} \right|}}$
and
\[\frac{1}{q}h_K^{1 - q}\left( \xi  \right)h_L^q\left( \xi  \right)
=\left\langle {\xi ,\nabla \left( {\frac{1}{q}h_K^{1 - q}
\left(\xi\right)h_L^q\left(\xi\right)} \right)} \right\rangle,\]
due to the Euler's homogeneous function theorem.
We can conclude that
\[\omega \left( x \right) = \left| {\nabla u\left( x \right)} \right|\left( {\frac{1}{q}h_K^{1 - q}\left( {{g_K}\left( x \right)} \right)h_L^q\left( {{g_K}\left( x \right)} \right)} \right).\]
This completes the proof of the second assertion for the case $0<q<1$.
\end{proof}

In the following, we prove two lemmas which are critical for establishing the variational formula of
$\Gamma\left(K \right)$ with respect to the $q$-sum. The first one can be stated as follows.
\begin{lemma}\label{lem:3.3}
Let $1<\mathbf{p}<\infty$, and let $K, L\in \mathcal{A}_+^{2,\alpha}$
be two compact convex sets containing the origin.
Then, for the Wulff shape $K^t$ with $\left| t \right|\le\tau$ (where $\tau$ is given in \eqref{3.4}),
if $0<q<1$, we have
\begin{equation*}
\begin{split}
{\left.{\frac{d}{{dt}}}\right|_{t = 0}}\mathcal{F}\left[ {{h_{{K^t}}}} \right]\left(\xi\right)
=&\sum\limits_{i,j=1}^{n-1} {{\bigtriangledown _j}
\left({{C_{i,j}}\left[ {{\bigtriangledown ^2}{h_K}+{h_K}{\mathbb I}}\right]
{{\left|{\nabla u\left({\nabla {h_{{K}}}\left(\xi\right)}\right)}\right|}^{\mathbf{p}-1}}
{\bigtriangledown _i}\left( {\frac{1}{q}h_K^{1-q}h_L^q} \right)}\right)}\\
&-\left( {\mathbf{p}-1}\right)
{\left| {\nabla u\left( {\nabla {h_{{K}}}\left(\xi\right)}\right)}\right|^{\mathbf{p}-2}}
\det\left( {{\bigtriangledown ^2}{h_K} + {h_K}{\mathbb I}}\right)
\left\langle{\nabla\dot u\left({\nabla{h_K}\left(\xi\right)}\right),\xi}\right\rangle
\end{split}
\end{equation*}
$S_K$-almost everywhere on $\mathbb{S}^{n-1}$. If $q\ge1$, this equality always holds on $\mathbb{S}^{n-1}$.
\end{lemma}

\begin{proof}	
Since the proof for the case $q\ge1$ is similar to that for the case $0<q<1$,
we will focus only on the latter.

According to \eqref{3.6}, we have the following calculation
\begin{equation}\label{3.11}
\begin{split}
&{\left.{\frac{d}{{dt}}}\right|_{t = 0}}\mathcal{F}\left[ {{h_{{K^t}}}}\right]\left(\xi\right)\\
=&{\left. {\frac{d}{{dt}}} \right|_{t = 0}}\left( {{{\left| {\nabla u\left( {\nabla {h_{{K^t}}}\left(\xi\right),t} \right)} \right|}^{\mathbf{p}-1}}\det \left( {{\bigtriangledown ^2}{h_{K^t}} + {h_{K^t}}{\mathbb I}} \right)} \right)\\
=&\left( {\mathbf{p}-1} \right){\left| {\nabla u\left( {\nabla {h_K}\left( \xi  \right)} \right)} \right|^{\mathbf{p}-2}}{\left. {\det \left( {{\bigtriangledown ^2}{h_K} + {h_K}{\mathbb I}} \right)\frac{d}{{dt}}} \right|_{t = 0}}\left| {\nabla u\left( {\nabla {h_{{K^t}}}\left( \xi  \right),t} \right)} \right| \\
&+{\left| {\nabla u\left( {\nabla {h_K}\left( \xi  \right)} \right)} \right|^{\mathbf{p}-1}}{\left. {\frac{d}{{dt}}} \right|_{t = 0}}\det \left( {{\bigtriangledown ^2}{h_{{K^t}}} + {h_{{K^t}}}{\mathbb I}} \right).		
\end{split}
\end{equation}
Notice that
\[\int_{{\mathbb{S}^{n-1}}} {\left( {{\bigtriangledown ^2}{h_{{K^t}}} + {h_{{K^t}}}{\mathbb{I}}} \right)} d{S_{{K^t}}}
= \int_{{\mathbb{S}^{n-1}}} \left({\bigtriangledown ^2}{{\left( {h_K^q + th_L^q} \right)}^{\frac{1}{q}}}
+ {{\left( {h_K^q + th_L^q} \right)}^{\frac{1}{q}}}{\mathbb{I}}\right) d{S_{{K^t}}},\]
we differentiate both sides with respect to $t$ at $t=0$ and obtain
\begin{equation*}
\begin{split}
&\int_{{\mathbb{S}^{n-1}}}{{{\left. {\frac{d}{{dt}}} \right|}_{t = 0}}\left( {{\bigtriangledown ^2}{h_{{K^t}}}
+ {h_{{K^t}}}{\mathbb{I}}} \right)} d{S_K}
+\int_{{\mathbb{S}^{n-1}}} \left({{\bigtriangledown ^2}{h_{{K}}} + {h_{{K}}}{\mathbb{I}}}\right)
{\left. {\frac{d}{{dt}}} \right|_{t = 0}}d{S_{{K^t}}}\\
=&\int_{{\mathbb{S}^{n-1}}} {{{\left. {\frac{d}{{dt}}} \right|}_{t = 0}}
\left( {{\bigtriangledown ^2}{{\left( {h_K^q + th_L^q} \right)}^{\frac{1}{q}}}
+ {{\left( {h_K^q + th_L^q} \right)}^{\frac{1}{q}}}{\mathbb{I}}} \right)} d{S_K}
+\int_{{\mathbb{S}^{n-1}}} \left({{\bigtriangledown ^2}{{ {h_K}}}
+ { {h_K} }{\mathbb{I}}}\right) {\left. {\frac{d}{{dt}}} \right|_{t = 0}}d{S_{{K^t}}}.
\end{split}
\end{equation*}
This implies that
\[\int_{{\mathbb{S}^{n-1}}} {{{\left. {\frac{d}{{dt}}} \right|}_{t = 0}}\left( {{\bigtriangledown ^2}{h_{{K^t}}} + {h_{{K^t}}}{\mathbb{I}}} \right)} d{S_K}
= \int_{{\mathbb{S}^{n-1}}} {{{\left. {\frac{d}{{dt}}} \right|}_{t = 0}}\left( {{\bigtriangledown ^2}{{\left( {h_K^q + th_L^q} \right)}^{\frac{1}{q}}} + {{\left( {h_K^q + th_L^q} \right)}^{\frac{1}{q}}}{\mathbb{I}}} \right)} d{S_K}.\]
Therefore,
\[{{{\left. {\frac{d}{{dt}}} \right|}_{t = 0}}\left( {{\bigtriangledown ^2}{h_{{K^t}}} + {h_{{K^t}}}{\mathbb{I}}} \right) = {{\left. {\frac{d}{{dt}}} \right|}_{t = 0}}\left( {{\bigtriangledown ^2}{{\left( {h_K^q + th_L^q} \right)}^{\frac{1}{q}}} + {{\left( {h_K^q + th_L^q} \right)}^{\frac{1}{q}}}{\mathbb{I}}} \right)}\]
$S_K$-almost everywhere. Hence,
 \begin{equation}\label{3.12}
\begin{split}
&{\left. {\frac{d}{{dt}}} \right|_{t = 0}}\det \left( {{\bigtriangledown ^2}{h_{{K^t}}} + {h_{{K^t}}}{\mathbb I}} \right)\\
=&{\rm{Tr}}\left( {C\left[ {{\bigtriangledown ^2}{h_K} + {h_K}{\mathbb I}} \right]{{\left. {\frac{d}{{dt}}} \right|}_{t = 0}}\left( {{\bigtriangledown ^2}{h_{{K^t}}} + {h_{{K^t}}}{\mathbb I}} \right)} \right)\\
=&{\rm{Tr}}\left( {C\left[ {{\bigtriangledown ^2}{h_K} + {h_K}{\mathbb I}} \right]{{\left. {\frac{d}{{dt}}} \right|}_{t = 0}}\left( {{\bigtriangledown ^2}{{\left( {h_K^q + th_L^q} \right)}^{\frac{1}{q}}} + {{\left( {h_K^q + th_L^q} \right)}^{\frac{1}{q}}}{\mathbb I}} \right)} \right)\\
=&{\rm{Tr}}\left( {C\left[ {{\bigtriangledown ^2}{h_K} + {h_K}{\mathbb I}} \right]\left( {{\bigtriangledown ^2}\left( {\frac{1}{q}h_K^{1 - q}h_L^q} \right) + \left( {\frac{1}{q}h_K^{1 - q}h_L^q} \right){\mathbb I}} \right)} \right).
\end{split}
\end{equation}
$S_K$-almost everywhere.

As the unit outer normal $\xi$ of $K^t$ satisfies the identity
\begin{equation*}
\xi  =  - \frac{{\nabla u\left( {\nabla {h_{{K^t}}}\left( \xi  \right),t} \right)}}{{\left| {\nabla u\left( {\nabla {h_{{K^t}}}\left( \xi  \right),t} \right)} \right|}},
\end{equation*}
then
$\left| {\nabla u\left( {\nabla {h_{{K^t}}}\left( \xi  \right),t} \right)} \right| =  - \left\langle {\nabla u\left( {\nabla {h_{{K^t}}}\left( \xi  \right),t} \right),\xi } \right\rangle$,
and we have the following calculation
\begin{equation*}
\begin{split}
&{\left. {\frac{d}{{dt}}} \right|_{t = 0}}\left| {\nabla u\left( {\nabla {h_{{K^t}}}\left( \xi  \right),t} \right)} \right|\\
=&-{\left. {\frac{d}{{dt}}} \right|_{t = 0}}\left\langle {\nabla u\left( {\nabla {h_{{K^t}}}\left( \xi  \right),t} \right),\xi } \right\rangle \\
=&-\left( {\left\langle {{D^2}u\left( {\nabla {h_K}\left( \xi  \right)} \right){{\left. {\frac{d}{{dt}}} \right|}_{t = 0}}\nabla {h_{{K^t}}}\left( \xi  \right),\xi } \right\rangle  + \left\langle {\nabla \dot u\left( {\nabla {h_K}\left( \xi  \right)} \right),\xi } \right\rangle } \right)\\
=&-\left( {\left\langle {{D^2}u\left( {\nabla {h_K}\left( \xi  \right)} \right)\nabla \left( {{{\left. {\frac{d}{{dt}}} \right|}_{t = 0}}{{\left( {h_K^q + th_L^q} \right)}^{\frac{1}{q}}}} \right),\xi } \right\rangle  + \left\langle {\nabla \dot u\left( {\nabla {h_K}\left( \xi  \right)} \right),\xi } \right\rangle } \right)\\
=&-\left( {\left\langle {{D^2}u\left( {\nabla {h_K}\left( \xi  \right)} \right)\nabla \left( {\frac{1}{q}h_K^{1 - q}h_L^q} \right),\xi } \right\rangle  + \left\langle {\nabla \dot u\left( {\nabla {h_K}\left( \xi  \right)} \right),\xi } \right\rangle } \right)\\
=&-\left( {{J_1} + {J_2}} \right),
\end{split}
\end{equation*}
$S_K$-almost everywhere.
Since
$$
\nabla {h_K}\left( \xi  \right) = {h_K}\left( \xi  \right)\xi  + \sum\limits_{i = 1}^{n - 1} {{\bigtriangledown _i}{h_K}\left( \xi  \right){e^i}}
$$
and
$$
\nabla {h_L}\left( \xi  \right) = {h_L}\left( \xi  \right)\xi  + \sum\limits_{i = 1}^{n - 1} {{\bigtriangledown _i}{h_L}\left( \xi  \right){e^i}},
$$
we have
\begin{equation}\label{3.13}
\begin{split}
\nabla \left( {\frac{1}{q}h_K^{1-q}(\xi)h_L^q(\xi)} \right)
=\left( {\frac{1}{q}h_K^{1-q}(\xi)h_L^q(\xi)} \right)\xi
+\sum\limits_{i=1}^{n-1} {{\bigtriangledown_i}
\left({\frac{1}{q}h_K^{1-q}(\xi)h_L^q(\xi)}\right){e^i}}.
\end{split}
\end{equation}
This, together with Lemma \ref{lem:2.2}, yields that
\begin{equation*}
\begin{split}
{J_1}=&\left\langle {{D^2}u\left( {\nabla {h_K}\left( \xi  \right)} \right)\nabla
\left( {\frac{1}{q}h_K^{1-q}h_L^q} \right),\xi } \right\rangle\\
=&\left\langle {{D^2}u\left( {\nabla {h_K}\left( \xi  \right)} \right)\xi ,\xi } \right\rangle \left( {\frac{1}{q}h_K^{1 - q}h_L^q} \right)
+\sum\limits_{i = 1}^{n - 1} {\left\langle {{D^2}u\left( {\nabla {h_K}\left(\xi\right)} \right){e^i},\xi} \right\rangle {\bigtriangledown _i}
\left( {\frac{1}{q}h_K^{1 - q}h_L^q} \right)} \\
=&\frac{1}{{\mathbf{p}-1}}\kappa \left( {\nabla {h_K}\left( \xi  \right)} \right)
\left|{\nabla u\left( {\nabla {h_K}\left( \xi  \right)} \right)} \right|
{\rm{Tr}}\left( {C\left[ {{\bigtriangledown ^2}{h_K}+{h_K}{\mathbb I}} \right]} \right)
\left( {\frac{1}{q}h_K^{1 - q}h_L^q} \right)\\
&-\sum\limits_{i = 1}^{n - 1} {\kappa \left( {\nabla {h_K}\left(\xi\right)} \right)\sum\limits_{j = 1}^{n - 1} {{C_{i,j}}\left[ {{\bigtriangledown ^2}{h_K}
+ {h_K}{\mathbb I}} \right]} {\bigtriangledown _j}
\left( {\left| {\nabla u\left( {\nabla {h_K}\left( \xi  \right)} \right)} \right|} \right){\bigtriangledown _i}\left( {\frac{1}{q}h_K^{1 - q}h_L^q} \right)} \\
=&\frac{1}{{\mathbf{p}-1}}\kappa \left( {\nabla {h_K}\left( \xi  \right)} \right)
\left| {\nabla u\left( {\nabla {h_K}\left( \xi  \right)} \right)} \right|
{\rm{Tr}}\left( {C\left[ {{\bigtriangledown ^2}{h_K}+{h_K}{\mathbb I}} \right]} \right)
\left( {\frac{1}{q}h_K^{1 - q}h_L^q} \right)\\
&-\kappa \left( {\nabla {h_K}\left(\xi\right)} \right)\sum\limits_{i,j = 1}^{n - 1} {{C_{i,j}}\left[ {{\bigtriangledown ^2}{h_K} + {h_K}{\mathbb I}} \right]} {\bigtriangledown _j}
\left( {\left| {\nabla u\left( {\nabla {h_K}
\left( \xi  \right)} \right)} \right|} \right){\bigtriangledown _i}
\left( {\frac{1}{q}h_K^{1 - q}h_L^q} \right).
\end{split}
\end{equation*}	
Then, using $\sum\limits_{j=1}^{n-1} {{\bigtriangledown _j}{C_{i,j}}\left[ {{\bigtriangledown ^2}{h_K} + {h_K}{\mathbb I}} \right]}=0$ (cf. (4.3) of \cite{CY76}), we have
\begin{equation*}
\begin{split}
{J_1} =& \frac{1}{{\mathbf{p}-1}}\kappa \left( {\nabla {h_K}\left( \xi  \right)} \right)\left| {\nabla u\left( {\nabla {h_K}\left( \xi  \right)} \right)} \right|{\rm{Tr}}\left( {C\left[ {{\bigtriangledown ^2}{h_K} + {h_K}{\mathbb I}} \right]} \right)\left( {\frac{1}{q}h_K^{1 - q}h_L^q} \right)\\
&- \kappa \left( {\nabla {h_K}\left( \xi  \right)} \right)\sum\limits_{i,j = 1}^{n - 1} {{\bigtriangledown _j}\left( {{C_{i,j}}\left[ {{\bigtriangledown ^2}{h_K} + {h_K}{\mathbb I}} \right]{\left| {\nabla u\left( {\nabla {h_K}\left( \xi  \right)} \right)} \right|}} \right){\bigtriangledown _i}\left( {\frac{1}{q}h_K^{1 - q}h_L^q} \right)}.
\end{split}
\end{equation*}
Hence,
\begin{equation}\label{3.14}
\begin{split}
&{\left. {\frac{d}{{dt}}} \right|_{t = 0}}\left| {\nabla u\left( {\nabla {h_{{K^t}}}\left( \xi  \right),t} \right)} \right| \\
=&\kappa \left( {\nabla {h_K}\left( \xi  \right)} \right)\sum\limits_{i,j = 1}^{n - 1} {{\bigtriangledown _j}\left( {{C_{i,j}}\left[ {{\bigtriangledown ^2}{h_K} + {h_K}{\mathbb I}} \right]\left( {\left| {\nabla u\left( {\nabla {h_K}\left( \xi  \right)} \right)} \right|} \right)} \right){\bigtriangledown _i}\left( {\frac{1}{q}h_K^{1 - q}h_L^q} \right)} \\
&-\frac{1}{{\mathbf{p}-1}}\kappa \left( {\nabla {h_K}\left( \xi  \right)} \right)\left| {\nabla u\left( {\nabla {h_K}\left( \xi  \right)} \right)} \right|{\rm{Tr}}\left( {C\left[ {{\bigtriangledown ^2}{h_K} + {h_K}{\mathbb I}} \right]} \right)\left( {\frac{1}{q}h_K^{1 - q}h_L^q} \right) \\
&-\left\langle {\nabla \dot u\left( {\nabla {h_K}\left( \xi  \right)} \right),\xi } \right\rangle,
\end{split}
\end{equation}
$S_K$-almost everywhere.

Applying \eqref{2.7} and substituting both \eqref{3.14} and \eqref{3.12} into \eqref{3.11},
we obtain that
\begin{equation*}
\begin{split}
&{\left. {\frac{d}{{dt}}} \right|_{t = 0}}\mathcal{F}\left[ {{h_{{K^t}}}} \right]\left( \xi\right)\\
=&\left( {\mathbf{p}-1} \right){\left| {\nabla u\left( {\nabla {h_K}\left( \xi  \right)} \right)} \right|^{\mathbf{p}-2}}\sum\limits_{i,j = 1}^{n - 1} {{\bigtriangledown _j}\left( {{C_{i,j}}\left[ {{\bigtriangledown ^2}{h_K} + {h_K}{\mathbb I}} \right]\left( {\left| {\nabla u\left( {\nabla {h_K}\left( \xi  \right)} \right)} \right|} \right)} \right){\bigtriangledown _i}\left( {\frac{1}{q}h_K^{1 - q}h_L^q} \right)}\\
&-{\left| {\nabla u\left( {\nabla {h_K}\left( \xi  \right)} \right)} \right|^{\mathbf{p}-1}}{\rm{Tr}}\left( {C\left[ {{\bigtriangledown ^2}{h_K} + {h_K}{\mathbb I}} \right]} \right)\left( {\frac{1}{q}h_K^{1 - q}h_L^q} \right) \\
&-\left( {\mathbf{p}-1} \right)\frac{{{{\left| {\nabla u\left( {\nabla {h_K}\left( \xi  \right)} \right)} \right|}^{\mathbf{p}-2}}}}{{\kappa \left( {\nabla {h_K}\left( \xi  \right)} \right)}}\left\langle {\nabla \dot u\left( {\nabla {h_K}\left( \xi  \right)} \right),\xi } \right\rangle \\
&+{\left| {\nabla u\left( {\nabla {h_K}\left( \xi  \right)} \right)} \right|^{\mathbf{p}-1}}{\rm{Tr}}\left( {C\left[ {{\bigtriangledown ^2}{h_K} + {h_K}{\mathbb I}} \right]\left( {{\bigtriangledown ^2}\left( {\frac{1}{q}h_K^{1 - q}h_L^q} \right) + \left( {\frac{1}{q}h_K^{1 - q}h_L^q} \right){\mathbb I}} \right)} \right)\\
=&\left( {\mathbf{p}-1} \right){\left| {\nabla u\left( {\nabla {h_K}\left( \xi  \right)} \right)} \right|^{\mathbf{p}-2}}\sum\limits_{i,j = 1}^{n - 1} {{\bigtriangledown _j}\left( {{C_{i,j}}\left[ {{\bigtriangledown ^2}{h_K} + {h_K}{\mathbb I}} \right]\left( {\left| {\nabla u\left( {\nabla {h_K}\left( \xi  \right)} \right)} \right|} \right)} \right){\bigtriangledown _i}\left( {\frac{1}{q}h_K^{1 - q}h_L^q} \right)} \\
&-\left( {\mathbf{p}-1} \right)\frac{{{{\left| {\nabla u\left( {\nabla {h_K}\left( \xi  \right)} \right)} \right|}^{\mathbf{p}-2}}}}{{\kappa \left( {\nabla {h_K}\left( \xi  \right)} \right)}}\left\langle {\nabla \dot u\left( {\nabla {h_K}\left( \xi  \right)} \right),\xi } \right\rangle \\
&+ {\left| {\nabla u\left( {\nabla {h_K}\left( \xi  \right)} \right)} \right|^{\mathbf{p}-1}}{\rm{Tr}}\left( {C\left[ {{\bigtriangledown ^2}{h_K} + {h_K}{\mathbb I}} \right]\left( {{\bigtriangledown ^2}\left( {\frac{1}{q}h_K^{1 - q}h_L^q} \right)} \right)} \right),
\end{split}
\end{equation*}
$S_K$-almost everywhere. Since
\begin{equation*}
\begin{split}
&\sum\limits_{i,j = 1}^{n - 1}{{\bigtriangledown _j}\left( {{C_{i,j}}\left[ {{\bigtriangledown ^2}{h_K} + {h_K}{\mathbb I}} \right]{{\left| {\nabla u\left( {\nabla {h_K}\left( \xi  \right)} \right)} \right|}^{\mathbf{p}-1}}{\bigtriangledown _i}\left( {\frac{1}{q}h_K^{1 - q}h_L^q} \right)} \right)}\\
=&\sum\limits_{i,j = 1}^{n - 1} {{\bigtriangledown _j}\left( {{C_{i,j}}\left[ {{\bigtriangledown ^2}{h_K} + {h_K}{\mathbb I}} \right]{{\left| {\nabla u\left( {\nabla {h_K}\left( \xi  \right)} \right)} \right|}^{\mathbf{p}-1}}} \right)} {\bigtriangledown _i}\left( {\frac{1}{q}h_K^{1 - q}h_L^q} \right)\\
&+\sum\limits_{i,j = 1}^{n - 1} {{C_{i,j}}\left[ {{\bigtriangledown ^2}{h_K} + {h_K}{\mathbb I}} \right]{{\left| {\nabla u\left( {\nabla {h_K}\left( \xi  \right)} \right)} \right|}^{\mathbf{p}-1}}} {\bigtriangledown _{j,i}}\left( {\frac{1}{q}h_K^{1 - q}h_L^q} \right)\\
=&\left( {\mathbf{p}-1} \right){\left| {\nabla u\left( {\nabla {h_K}\left( \xi  \right)} \right)} \right|^{\mathbf{p}-2}}\sum\limits_{i,j = 1}^{n - 1} {{\bigtriangledown _j}\left( {{C_{i,j}}\left[ {{\bigtriangledown ^2}{h_K} + {h_K}{\mathbb I}} \right]\left( {\left| {\nabla u\left( {\nabla {h_K}\left( \xi  \right)} \right)} \right|} \right)} \right){\bigtriangledown _i}\left( {\frac{1}{q}h_K^{1 - q}h_L^q} \right)}\\
&+ {\left| {\nabla u\left( {\nabla {h_K}\left( \xi  \right)} \right)} \right|^{\mathbf{p}-1}}{\rm{Tr}}\left( {C\left[ {{\bigtriangledown ^2}{h_K} + {h_K}{\mathbb I}} \right]{\bigtriangledown ^2}\left( {\frac{1}{q}h_K^{1 - q}h_L^q} \right)} \right).
\end{split}
\end{equation*}
 Hence,
\begin{equation*}
\begin{split}
{\left. {\frac{d}{{dt}}} \right|_{t = 0}}\mathcal{F}\left[ {{h_{{K^t}}}} \right]\left( \xi  \right)
=& \sum\limits_{i,j = 1}^{n - 1}
{{\bigtriangledown _j}\left( {{C_{i,j}}\left[ {{\bigtriangledown ^2}{h_K} + {h_K}{\mathbb I}} \right]{{\left| {\nabla u\left( {\nabla {h_K}\left( \xi  \right)} \right)} \right|}^{\mathbf{p}-1}}{\bigtriangledown _i}\left( {\frac{1}{q}h_K^{1 - q}h_L^q} \right)} \right)} \\
&-\left( {\mathbf{p}-1} \right)\frac{{{{\left| {\nabla u\left( {\nabla {h_K}\left( \xi  \right)} \right)} \right|}^{\mathbf{p}-2}}}}{{\kappa \left( {\nabla {h_K}\left( \xi  \right)} \right)}}\left\langle {\nabla \dot u\left( {\nabla {h_K}\left( \xi  \right)} \right),\xi } \right\rangle,
\end{split}
\end{equation*}
$S_K$-almost everywhere.
\end{proof}

Lemmas \ref{lem:3.2} and \ref{lem:3.3} can be employed to prove the following result.

\begin{lemma}\label{lem:3.4}
Let $1<\mathbf{p}<\infty$ and $q>0$, and let $K, L\in \mathcal{A}_+^{2,\alpha}$
be two compact convex sets containing the origin.
Then, for the Wulff shape $K^t$ with $\left| t \right|\le\tau$ (where $\tau$ is given in \eqref{3.4}), we have
\begin{equation}\label{3.15}
\int_{{\mathbb{S}^{n-1}}} {{h_K}{{\left. {\frac{d}{{dt}}} \right|}_{t = 0}}\mathcal{F}\left[ {{h_{{K^t}}}} \right]\left( \xi  \right)} d\xi  = \int_{{\mathbb{S}^{n-1}}} {h_K^{1 - q}h_L^q{{\left. {\frac{d}{{dt}}} \right|}_{t = 0}}\mathcal{F}\left[ {{{\left( {1 + t} \right)}^{\frac{1}{q}}}{h_K}} \right]\left( \xi  \right)} d\xi.
\end{equation}
\end{lemma}

\begin{proof}
Since $K\in\mathcal{A}_+^{2,\alpha}$, by Lemma \ref{lem:3.3},
we have
\begin{equation}\label{3.16}
\begin{split}
&\int_{{\mathbb{S}^{n-1}}}{{h_K}{{\left. {\frac{d}{{dt}}} \right|}_{t = 0}}\mathcal{F}\left[ {{h_{{K^t}}}} \right]\left( \xi  \right)} d\xi \\
=&\int_{{\mathbb{S}^{n-1}}} {{h_K}\sum\limits_{i,j = 1}^{n - 1} {{\bigtriangledown _j}\left( {{C_{i,j}}\left[ {{\bigtriangledown ^2}{h_K} + {h_K}{\mathbb I}} \right]{{\left| {\nabla u\left( {\nabla {h_K}\left( \xi  \right)} \right)} \right|}^{\mathbf{p}-1}}{\bigtriangledown _i}\left( {\frac{1}{q}h_K^{1 - q}h_L^q} \right)} \right)} } d\xi \\
&-\int_{{\mathbb{S}^{n-1}}} {{h_K}\left( {\mathbf{p}-1} \right)\frac{{{{\left| {\nabla u\left( {\nabla {h_K}\left( \xi  \right)} \right)} \right|}^{\mathbf{p}-2}}}}{{\kappa \left( {\nabla {h_K}\left( \xi  \right)} \right)}}\left\langle {\nabla \dot u\left( {\nabla {h_K}\left( \xi  \right)} \right),\xi } \right\rangle } d\xi \\
=&I_1-I_2.
\end{split}
\end{equation}
Then, by repeatedly applying Stokes's theorem for a compact manifold without boundary, we can calculate the term $I_1$ as follows.
\begin{equation}\label{3.17}
\begin{split}
I_1
& = \int_{{\mathbb{S}^{n-1}}} {\sum\limits_{i,j = 1}^{n - 1} {{h_K}{\bigtriangledown _j}\left( {{C_{i,j}}\left[ {{\bigtriangledown ^2}{h_K} + {h_K}{\mathbb{I}}} \right]{{\left| {\nabla u\left( {\nabla {h_K}\left( \xi  \right)} \right)} \right|}^{\mathbf{p}-1}}{\bigtriangledown _i}\left( {\frac{1}{q}h_K^{1 - q}h_L^q} \right)} \right)} } d\xi \\
&=-\int_{{\mathbb{S}^{n-1}}} {\sum\limits_{i,j = 1}^{n - 1} {{C_{i,j}}\left[ {{\bigtriangledown^2}{h_K} + {h_K}{\mathbb{I}}} \right]{{\left| {\nabla u\left( {\nabla {h_K}\left( \xi  \right)} \right)} \right|}^{\mathbf{p}-1}}{\bigtriangledown _i}\left( {\frac{1}{q}h_K^{1 - q}h_L^q} \right){\bigtriangledown _j}} } {h_K}d\xi\\
&=\int_{{\mathbb{S}^{n-1}}} {\sum\limits_{i,j = 1}^{n - 1} {h_K^{1-q}h_L^q{\bigtriangledown _j}\left( {{C_{i,j}}\left[ {{\bigtriangledown ^2}{h_K} + {h_K}{\mathbb{I}}} \right]{{\left| {\nabla u\left( {\nabla {h_K}\left( \xi  \right)} \right)} \right|}^{\mathbf{p}-1}}{\bigtriangledown _i}\left( {\frac{1}{q}{h_K}} \right)} \right)} } d\xi.
\end{split}
\end{equation}
By using (ii) of Lemma \ref{lem:3.2}, along with the formulas \eqref{3.13} and \eqref{2.2}, we can calculate
\begin{equation*}
\begin{split}
\frac{1}{\mathbf{p}-1}{I_2}
=&\int_{{\mathbb{S}^{n-1}}} {{h_K}\frac{{{{\left| {\nabla u\left( {\nabla {h_K}\left( \xi  \right)} \right)} \right|}^{\mathbf{p}-2}}}}{{\kappa \left( {\nabla {h_K}\left( \xi  \right)} \right)}}\left\langle {\nabla \dot u\left( {\nabla {h_K}\left( \xi  \right)} \right),\xi } \right\rangle } d\xi \\
=&\int_{\partial K} {{{\left| {\nabla u} \right|}^{\mathbf{p}-2}}{h_K} \circ {g_K}\left\langle {\nabla \left( {\left| {\nabla u} \right|\left( {\frac{1}{q}{{\left( {{h_K} \circ {g_K}} \right)}^{1 - q}}{{\left( {{h_L} \circ {g_K}} \right)}^q}} \right)} \right),{g_K}} \right\rangle } d{\mathcal{H}^{n - 1}}\\
=&\int_{\partial K} {{{\left| {\nabla u} \right|}^{\mathbf{p}-2}}{h_K} \circ {g_K}\left\langle {\nabla \left( {\left| {\nabla u} \right|} \right)\left( {\frac{1}{q}{{\left( {{h_K} \circ {g_K}} \right)}^{1 - q}}{{\left( {{h_L} \circ {g_K}} \right)}^q}} \right),{g_K}} \right\rangle } d{\mathcal{H}^{n - 1}}\\
&+\int_{\partial K} {{{\left| {\nabla u} \right|}^{\mathbf{p}-2}}{h_K} \circ {g_K}\left| {\nabla u} \right|\frac{1}{q}{{\left( {{h_K} \circ {g_K}} \right)}^{1 - q}}{{\left( {{h_L} \circ {g_K}} \right)}^q}} d{\mathcal{H}^{n - 1}}\\
=&\int_{\partial K} {{{\left| {\nabla u} \right|}^{\mathbf{p}-2}}{{\left( {{h_K} \circ {g_K}} \right)}^{1 - q}}{{\left( {{h_L} \circ {g_K}} \right)}^q}\left\langle {\nabla \left( {\left| {\nabla u} \right|} \right)\frac{1}{q}{h_K} \circ {g_K},{g_K}} \right\rangle } d{\mathcal{H}^{n - 1}}\\
&+\int_{\partial K} {{{\left| {\nabla u} \right|}^{\mathbf{p}-2}}{{\left( {{h_K} \circ {g_K}} \right)}^{1 - q}}{{\left( {{h_L} \circ {g_K}} \right)}^q}\left\langle {\left| {\nabla u} \right|\nabla \left( {\frac{1}{q}{h_K} \circ {g_K}} \right),{g_K}} \right\rangle } d{\mathcal{H}^{n - 1}}\\
=&\int_{{\mathbb{S}^{n-1}}} {h_K^{1 - q}h_L^q\frac{{{{\left| {\nabla u\left( {\nabla {h_K}\left( \xi  \right)} \right)} \right|}^{\mathbf{p}-2}}}}{{\kappa \left( {\nabla {h_K}\left( \xi  \right)} \right)}}\left\langle {\nabla \left( {\left| {\nabla u\left( {\nabla {h_K}\left( \xi  \right)} \right)} \right|\left( {\frac{1}{q}{h_K}} \right)} \right),\xi } \right\rangle } d\xi.
\end{split}
\end{equation*}
This, together with \eqref{3.17} and \eqref{3.16}, yields that
\begin{equation}\label{3.18}
\begin{split}
&\int_{{\mathbb{S}^{n-1}}}{{h_K}{{\left. {\frac{d}{{dt}}} \right|}_{t = 0}}\mathcal{F}\left[ {{h_{{K^t}}}} \right]\left( \xi  \right)} d\xi\\
=& \int_{{\mathbb{S}^{n-1}}} {h_K^{1 - q}h_L^q\sum\limits_{i,j = 1}^{n - 1} {{\bigtriangledown _j}\left( {{C_{i,j}}\left[ {{\bigtriangledown ^2}{h_K} + {h_K}{\mathbb{I}}} \right]{{\left| {\nabla u\left( {\nabla {h_K}\left( \xi  \right)} \right)} \right|}^{\mathbf{p}-1}}{\bigtriangledown _i}\left( {\frac{1}{q}{h_K}} \right)} \right)} d\xi } \\
&- \left( {\mathbf{p}-1} \right)\int_{{\mathbb{S}^{n-1}}} {h_K^{1 - q}h_L^q\frac{{{{\left| {\nabla u\left( {\nabla {h_K}\left( \xi  \right)} \right)} \right|}^{\mathbf{p}-2}}}}{{\kappa \left( {\nabla {h_K}\left( \xi  \right)} \right)}}\left\langle {\nabla \left( {\left| {\nabla u\left( {\nabla {h_K}\left( \xi  \right)} \right)} \right|\left( {\frac{1}{q}{h_K}} \right)} \right),\xi } \right\rangle } d\xi.
\end{split}
\end{equation}

On the other hand, by Lemma \ref{lem:3.2} and Lemma \ref{lem:3.3} with $L=K$, we have
\begin{equation*}
\begin{split}
&{\left. {\frac{d}{{dt}}} \right|_{t = 0}}{\cal F}\left[ {{{\left( {1 + t} \right)}^{\frac{1}{q}}}{h_K}} \right]\left( \xi  \right)\\
=& \sum\limits_{i,j = 1}^{n - 1} {{\bigtriangledown _j}\left( {{C_{i,j}}\left[ {{\bigtriangledown ^2}{h_K} + {h_K}{\mathbb{I}}} \right]{{\left| {\nabla u\left( {\nabla {h_K}\left( \xi  \right)} \right)} \right|}^{\mathbf{p}-1}}{\bigtriangledown _i}\left( {\frac{1}{q}h_K} \right)} \right)} \\
&-\left( {\mathbf{p}-1} \right)\frac{{{{\left| {\nabla u\left( {\nabla {h_K}\left( \xi  \right)} \right)} \right|}^{\mathbf{p}-2}}}}{{\kappa \left( {\nabla {h_K}\left( \xi  \right)} \right)}}\left\langle {\nabla \left( {\left| {\nabla u\left( {\nabla {h_K}\left( \xi  \right)} \right)} \right|\left( {\frac{1}{q}{h_K}} \right)} \right),\xi } \right\rangle,
\end{split}
\end{equation*}
for $q\ge1$. Note that the above equality holds almost everywhere with respect to
$S_K$ if $0<q<1$, then by substituting it into \eqref{3.18}, we can obtain \eqref{3.15}.
\end{proof}

Now, the main result of this section can be stated as follows.

\begin{theorem}\label{th:3.1}
Let $1<\mathbf{p}<\infty$, $q>0$, $K\in\mathcal K_o^n$ and $L\subset \mathbb{R}^{n}$ be a compact convex set containing the origin.
Then, for the Wulff shape $K^t$ with $\left| t \right|\le\tau$ (where $\tau$ is given in \eqref{3.4}), we have
\begin{equation}\label{3.19}
\begin{split}
{\left. {\frac{d}{{dt}}} \right|_{t = 0}}\Gamma \left( {{K^t}} \right) = \frac{{n-\mathbf{p} + 1}}{q}\int_{{\mathbb{S}^{n-1}}}
{h_L^q\left( \xi  \right)h_K^{1 - q}\left( \xi  \right)} d\mu_{K}\left(\xi\right) .
\end{split}
\end{equation}
\end{theorem}

\begin{proof}
Let $K\in\mathcal K_o^n$ and $L\subset \mathbb{R}^{n}$ be a compact convex set containing the origin. We first prove the case that $K, L\in \mathcal{A}_+^{2,\alpha}$.
Then, by formula \eqref{3.7} and Lemmas \ref{lem:3.4} and \ref{lem:3.1}, we have
\begin{equation*}
\begin{split}
&{\left. {\frac{d}{{dt}}} \right|_{t = 0}}\Gamma \left( {{K^t}} \right) \\
=& {\left. {\frac{d}{{dt}}} \right|_{t = 0}}\int_{{\mathbb{S}^{n-1}}} {{{\left( {h_K^q\left( \xi  \right) + th_L^q\left( \xi  \right)} \right)}^{\frac{1}{q}}}\mathcal{F} \left[ {{h_{{K^t}}}} \right]\left( \xi  \right)} d\xi \\
=& \int_{{\mathbb{S}^{n-1}}} {{{\left. {\frac{d}{{dt}}} \right|}_{t = 0}}{{\left( {h_K^q\left( \xi  \right) + th_L^q\left( \xi  \right)} \right)}^{\frac{1}{q}}}\mathcal{F} \left[ {{h_{{K^t}}}} \right]\left( \xi  \right)} d\xi  + \int_{{\mathbb{S}^{n-1}}} {{h_K}\left( \xi  \right){{\left. {\frac{d}{{dt}}} \right|}_{t = 0}}\mathcal{F} \left[ {{h_{{K^t}}}} \right]\left( \xi  \right)} d\xi \\
=& \frac{1}{q}\int_{{\mathbb{S}^{n-1}}}
{h_L^qh_K^{1-q}\mathcal{F} \left[ {{h_K}} \right]\left( \xi  \right)} d\xi
+\int_{{\mathbb{S}^{n-1}}} {h_L^q h_K^{1 - q}{{\left. {\frac{d}{{dt}}} \right|}_{t = 0}}\mathcal{F} \left[ {{{\left( {1 + t} \right)}^{\frac{1}{q}}}{h_K}} \right]\left( \xi  \right)} d\xi \\
=&\frac{1}{q}\int_{{\mathbb{S}^{n-1}}}
{h_L^qh_K^{1-q}\mathcal{F} \left[ {{h_K}} \right]\left( \xi  \right)} d\xi
+\frac{{n-\mathbf{p}}}{q}\int_{{\mathbb{S}^{n-1}}}
{h_L^qh_K^{1-q}\mathcal{F}\left[ {{h_K}} \right]\left( \xi  \right)} d\xi \\
=&\frac{{n-\mathbf{p} + 1}}{q}\int_{{\mathbb{S}^{n-1}}} {h_L^qh_K^{1 - q}\mathcal{F} \left[ {{h_K}} \right]\left( \xi  \right)}d\xi \\
=&\frac{{n-\mathbf{p} + 1}}{q}\int_{{\mathbb{S}^{n-1}}}
{h_L^qh_K^{1-q}}d\mu_{K}.
\end{split}
\end{equation*}
This proves \eqref{3.19} for the case
that $K, L\in \mathcal{A}_+^{2,\alpha}$.

For $K\in\mathcal K_o^n$ and a compact convex set $L \subset \mathbb{R}^{n}$ containing the origin, we can respectively choose two sequences
$\{K_j\}_{j=1}^\infty$ and $\{L_j\}_{j=1}^\infty$ in $\mathcal{A}_+^{2,\alpha}$, such that $K_j\to K$ and $L_j\to L$ as $j\to\infty$. It follows that $h_{K_j}\to h_K$ and  $h_{L_j}\to h_L$
uniformly. Then, by \eqref{2.3}, the continuity of the functional $\Gamma$ on compact convex sets and the weak convergence \eqref{2.8}, we can verify the
desired \eqref{3.19}.
\end{proof}

In view of the variational formula \eqref{3.19}, one can generalize the $\mathbf{p}$-harmonic measure
and introduce the following $L_{q}~\mathbf{p}$-harmonic measure.

\begin{definition}\label{def:3.1}
Let $q\in\mathbb{R}$, $1<\mathbf{p}<\infty$, and $K\in\mathcal K_o^n$.
We define the $L_q$ $\mathbf{p}$-harmonic measure $\mu _{K,q}$ for each Borel
$E\subset \mathbb{S}^{n-1}$ as
\begin{equation}\nonumber
{\mu _{K,q}}\left(E\right) = \int_E  {h_K^{1 - q}\left( \xi  \right)} d{\mu _K}\left(\xi\right).
\end{equation}
\end{definition}

The weak convergence of the $L_q$ $\mathbf{p}$-harmonic measure is critical and can be stated as follows.

\begin{lemma}\label{lem:3.5}
Let $q\in\mathbb{R}$, $1<\mathbf{p}<\infty$, and $K\in\mathcal K_o^n$.
Then for any sequence of convex bodies $\{K_j\}$ in $\mathcal K_o^n$,
if $K_{j}\to K$ as $j\to\infty$, then $\mu_{K_j,q}$ converges to $\mu_{K,q}$ weakly, as $j\to\infty$.
\end{lemma}

\begin{proof}
It follows from \eqref{2.8} that the $\mathbf{p}$-harmonic measure is convergent weakly. Then, by Definition \ref{def:3.1} and  $K_{j}\to K$ as $j\to\infty$, for any function $f\in C\left(\mathbb{S}^{n-1}\right)$, we have
\begin{equation}\nonumber
\lim\limits_{j\to\infty}\int_{\mathbb{S}^{n-1}}fd\mu_{K_j,q}
=\lim\limits_{j\to\infty}
\int_{\mathbb{S}^{n-1}}fh_{K_j}^{1-q}d\mu_{K_j}
=\int_{\mathbb{S}^{n-1}}fh_K^{1-q}d\mu_K
=\int_{\mathbb{S}^{n-1}}fd\mu_{K,q}.
\end{equation}
Thus, the desired weak convergence follows.
\end{proof}

\section{The proof of Theorem \ref{th:1.1}}\label{sect:4}

In this section, we study the $L_q$ Minkowski problem associated with $\mathbf{p}$-harmonic measure for $0<q<1$
and $1<\mathbf{p}\ne n+1$. By introducing an appropriate functional
and studying a related extremal problem as well as the existence of a solution,
we can finally prove Theorem \ref{th:1.1} via the variation method.
To begin with, we prove the following lemma, which is critical for our later approximation argument.

\begin{lemma}\label{lem:4.1}
Let $0<q<1$. If $f:\mathbb{S}^{n-1}\to \mathbb{R}$ is a positive, twice continuously
differentiable function, there exists a convex body $L$ containing the origin in its interior and a constant $r>0$ such that
$$f^q=h_L^q-h^q_{rB_2^n},$$
where $B_2^n$ is the standard unit ball in $\mathbb{R}^n$.
\end{lemma}

\begin{proof}
We extend the function $f$ to $\mathbb{R}^{n}\setminus\left\{o\right\}$
by defining $F\left(x\right):=\left|x\right|f\left(\frac{x}{\left|x\right| }\right)$
and we define $G\left(x\right):=\left|x\right|$ for $x\in\mathbb{R}^n$. Then, we can verify that
the function $\left(F^q+r^qG^q\right)^{\frac{1}{q}}$ is positively homogeneous of degree one,
where $r>0$. According to Euler's homogeneous function theorem,
\[\left\langle {x,\nabla{{{\left( {{F^q} + {r^q}{G^q}} \right)}^{\frac{1}{q}}}}} \right\rangle
= {\left( {{F^q} + {r^q}{G^q}} \right)^{\frac{1}{q}}},\]
we then take the first derivative with respect to each component $x_j$ of $x$ and obtain
\[\sum\limits_{i = 1}^n {\left( {\frac{{\partial {x_i}}}{{\partial {x_j}}}
\frac{{\partial \left( {{{\left( {{F^q} + {r^q}{G^q}} \right)}^{\frac{1}{q}}}} \right)}}{{\partial {x_i}}}
+ {x_i}\frac{{{\partial ^2}\left( {{{\left( {{F^q} + {r^q}{G^q}} \right)}^{\frac{1}{q}}}} \right)}}
{{\partial {x_i}\partial {x_j}}}} \right)}
= \frac{{\partial \left( {{{\left( {{F^q}+{r^q}{G^q}} \right)}^{\frac{1}{q}}}} \right)}}{{\partial {x_j}}},\]
where $j=1,\ldots,n$. Thus, we have
\begin{equation}\label{4.1}
\sum\limits_{i = 1}^n {\left( {{x_i}\frac{{{\partial ^2}\left( {{{\left( {{F^q} + {r^q}{G^q}} \right)}^{\frac{1}{q}}}} \right)}}{{\partial {x_i}\partial {x_j}}}} \right)}
=0,
\end{equation}
for all $j=1,\ldots,n$. Let $D^2_{x}\left(\left(F^q+r^qG^q\right)^{\frac{1}{q}}\right)$
be the second differential of function ${F^q}+{r^q}{G^q}$ at $x$,
that is
\[D_x^2\left( {{{\left( {{F^q} + {r^q}{G^q}} \right)}^{\frac{1}{q}}}} \right)
= {\left( {\frac{{{\partial ^2}\left( {{{\left( {{F^q} + {r^q}{G^q}} \right)}^{\frac{1}{q}}}} \right)}}{{\partial {x_i}\partial {x_j}}}} \right)_{ij}}.\]
It follows from \eqref{4.1} that
\begin{equation}\label{4.2}
xD^2_{x}\left(\left(F^q+r^qG^q\right)^{\frac{1}{q}}\right)z^\intercal=0,
\end{equation}
where $z^\intercal$ is the transpose of $z\in\mathbb{R}^{n}$.

For any two vectors $x,y\in \mathbb{S}^{n-1}$ with $x\perp y$, we can verify
\[yD_x^2\left( {{F^q}+{r^q}{G^q}}\right)y^\intercal
=yD_x^2\left( {{F^q}} \right)y^\intercal+q{r^q}.\]
Since the second differential $D_x^2\left( {{F^q}} \right)$ of
function $F^q$ is continuous on $\mathbb{S}^{n-1}$,
and $yD_x^2\left({{F^q}} \right)y^\intercal$ has a minimum, we can choose a suitable $r>0$ so that
\begin{equation}\label{4.3}
yD_x^2\left( {{ {{F^q} + {r^q}{G^q}}}} \right)y^\intercal
\ge 0.
\end{equation}

Let $x\in \mathbb{S}^{n-1}$. Then for any nonzero $z\in\mathbb{R}^{n}$, there exists $\alpha_1,\alpha_2\in\mathbb{R}$
such that $z=\alpha_1x+\alpha_2x'$, where $x'\perp x$ and $x'\in \mathbb{S}^{n-1}$. Since
\begin{equation*}
\begin{split}
&D_x^2\left( {{{\left( {{F^q} + {r^q}{G^q}} \right)}^{\frac{1}{q}}}} \right)\\
=&\frac{1}{q}\left( {\frac{1}{q} - 1} \right){\left( {{F^q} + {r^q}{G^q}} \right)^{\frac{1}{q}-2}}|{\nabla}\left( {{F^q}+{r^q}{G^q}} \right)|^2\mathrm{I}
+\frac{1}{q}{\left( {{F^q} + {r^q}{G^q}} \right)^{\frac{1}{q} - 1}}D_x^2\left( {{F^q} + {r^q}{G^q}} \right),
\end{split}
\end{equation*}
where $\mathrm{I}$ is the unit matrix of order $n$.
This, together with \eqref{4.2} and \eqref{4.3}, shows that
\begin{equation}\nonumber
zD^2_{x}\left(\left(F^q+r^qG^q\right)^{\frac{1}{q}}\right)z^\intercal
\ge 0,
\end{equation}
for any nonzero $z\in\mathbb{R}^{n}$ and $x\in\mathbb{S}^{n-1}$.
It follows that the matrix
$D_x^2\left( {{{\left( {{F^q} + {r^q}{G^q}} \right)}^{\frac{1}{q}}}} \right)$ is positive semi-definite
for any nonzero $x\in\mathbb{R}^{n}$. Then, by Theorem 1.5.13 of \cite{S14},
we can verify that the function ${{{\left( {{F^q} + {r^q}{G^q}} \right)}^{\frac{1}{q}}}}$ is sublinear.
The existence of the convex body $L$ directly follows from Theorem 1.7.1 of \cite{S14}.
\end{proof}

Let $Q$ be a compact convex set, $\mu$ be a finite Borel measure on $\mathbb{S}^{n-1}$, and $0<q<1$.
We define the functional
$\Phi_Q:Q\to\mathbb{R}$ as follows:
\begin{equation}\label{4.4}
{\Phi_Q}\left(\zeta\right)=\int_{{\mathbb{S}^{n-1}}} {{{\left( {{h_Q}\left(\xi\right)-\left\langle {\zeta,\xi}\right\rangle}\right)}^{q}}}d\mu\left(\xi\right).
\end{equation}
Next, we proceed to prove two necessary lemmas concerning the functional $\Phi_Q$.

\begin{lemma}\label{lem:4.2}
Let $0<q<1$ and $Q$ be a compact convex set, there exists a unique
$\zeta\left(Q\right)\in \rm{int} Q$ such that
\[{\Phi_Q}\left( {\zeta \left( Q \right)}\right)
=\mathop {\sup }\limits_{\zeta\in Q} {\Phi_Q}
\left(\zeta\right),\]
and for any $x_0\in\mathbb{R}^{n}$, we have ${\zeta\left(Q+x_0\right)}={\zeta\left(Q\right)}+x_0$.
\end{lemma}

\begin{proof}
Let $0<\lambda<1$ and $\zeta_{1},\zeta_{2}\in Q$.
From equality \eqref{4.4} and the concavity of the function $s^q$ with $s\ge0$ and $0<q<1$, we obtain that
\begin{equation*}
\begin{split}
&\lambda {\Phi_Q}\left( {{\zeta_1}} \right)
+\left( {1-\lambda } \right){\Phi_Q}\left( {{\zeta _2}} \right)\\
=&\int_{{\mathbb{S}^{n-1}}} {\lambda {{\left( {{h_Q}\left( \xi  \right) - \left\langle {{\zeta _1},\xi } \right\rangle } \right)}^q} + } \left( {1 - \lambda } \right){\left( {{h_Q}\left( \xi  \right) - \left\langle {{\zeta _2},\xi } \right\rangle } \right)^q}d\mu \left( \xi  \right)\\
\le&{\int_{{\mathbb{S}^{n-1}}} {\left( {{h_Q}\left( \xi  \right) - \left( {\lambda \left\langle {{\zeta _1},\xi } \right\rangle  + \left( {1 - \lambda } \right)\left\langle {{\zeta _2},\xi } \right\rangle } \right)} \right)} ^q}d\mu \left( \xi  \right)\\
=&{\Phi_Q}\left( {\lambda {\zeta _1} + \left( {1 - \lambda } \right){\zeta _2}} \right),
\end{split}
\end{equation*}
where the equality holds if and only if
$\left\langle{{\zeta_1},\xi}\right\rangle
=\left\langle{{\zeta_2},\xi}\right\rangle$
for all ${\xi\in \mathbb{S}^{n-1}}$, implying
${{\zeta _1} = {\zeta _2}}$.
Therefore, ${\Phi_Q}$ is strictly concave on $Q$, it follows that there exists a unique point
${\zeta\left( Q \right)}\in Q$ such that
${\Phi_Q}\left({\zeta\left(Q\right)}\right)
=\sup\limits_{\zeta\in Q}{\Phi_Q}\left(\zeta\right)$.
	
Next, we prove $\zeta\left(Q\right)\in \text{int}Q$. Suppose
to the contrary that $\zeta\left(Q\right)\in \partial Q$, and let $\omega$ be the set of all unit outward normal vectors at $\zeta\left(Q\right)$:
\[\omega = \left\{ {\left. {\xi  \in {\mathbb{S}^{n-1}}} \right|{h_Q}\left( \xi  \right)
=\left\langle {\zeta \left( Q \right),\xi } \right\rangle } \right\}.\]
Take $x_{0}\in \text{int}Q$ and define
\[{\xi _0} := \frac{{{x_0} - \zeta \left( Q \right)}}{{\left| {{x_0} - \zeta \left( Q \right)} \right|}}.\]
It can be verified that ${\left\langle {{\xi_0},\xi} \right\rangle }<0$ for $\xi \in \omega$.
Define
\begin{equation}\nonumber
\omega_+
:=\left\{\left.{\xi  \in{\mathbb{S}^{n-1}}\setminus\omega}\right|\left\langle{{\xi _0},\xi }\right\rangle \ge 0\right\}
\ \text{and}\
\omega_-
:=\left\{\left.{\xi  \in{\mathbb{S}^{n-1}}\setminus\omega}\right|\left\langle{{\xi _0},\xi }\right\rangle < 0\right\},
\end{equation}
then for $\xi\in\omega_+$, there exists a $\epsilon>0$ such that
${{h_Q}\left(\xi\right)-\left\langle{\zeta\left(Q \right),\xi} \right\rangle}
\ge\epsilon$. Choose $0<\delta<\frac{\epsilon}{2}$ small enough
so that $\zeta\left(Q\right)+\delta \xi_0\in\text{int}Q$, which further gives
\begin{equation}\nonumber
{h_Q}\left(\xi\right)-\left\langle {\zeta\left(Q \right)+\delta{\xi _0},\xi}\right\rangle
>\frac{{\epsilon}}{2},
\end{equation}
for $\xi\in\omega_+$. These, together with \eqref{4.4} and the Lagrange mean value theorem, imply that

\begin{equation}\nonumber
\begin{split}
&{\Phi_Q}\left({\zeta\left(Q\right)+\delta{\xi_0}}\right)- {\Phi_Q}\left({\zeta\left(Q\right)}\right)\\
=&\int_{{\mathbb{S}^{n-1}}} {{{\left( {{h_Q}\left( \xi  \right)-\left\langle{\zeta\left(Q\right)+\delta{\xi_0},\xi } \right\rangle}\right)}^q}}d\mu \left(\xi\right) - \int_{{\mathbb{S}^{n-1}}} {{{\left({{h_Q}\left(\xi\right)-
\left\langle {\zeta \left(Q\right),\xi}\right\rangle } \right)}^q}}d\mu\left(\xi\right)\\
=&\int_\omega{{{\left({-\left\langle {\delta {\xi_0},\xi }\right\rangle }\right)}^q}} d\mu \left(\xi  \right)+\int_{{\mathbb{S}^{n-1}}\setminus\omega} {{{\left( {{h_Q}\left(\xi\right)-\left\langle {\zeta \left(Q \right)
+\delta {\xi_0},\xi } \right\rangle } \right)}^q} - {{\left({{h_Q}\left(\xi\right)-\left\langle {\zeta\left( Q \right),\xi}\right\rangle}\right)}^q}} d\mu \left(\xi\right)\\
\ge&\int_\omega{{{\left({-\left\langle {\delta{\xi_0},\xi } \right\rangle}\right)}^q}} d\mu \left(\xi \right)
-\int_{{\omega_+}}{{{\left({{h_Q}\left(\xi  \right)-\left\langle{\zeta \left(Q\right),\xi} \right\rangle}\right)}^q}
-{{\left( {{h_Q}\left(\xi \right)- \left\langle{\zeta \left( Q \right)+\delta {\xi _0},\xi } \right\rangle}\right)}^q}}
d\mu\left(\xi \right)\\
>&\int_\omega{{{\left({-\left\langle{\delta {\xi _0},\xi } \right\rangle}\right)}^q}}
d\mu \left(\xi  \right)-\int_{{\omega_+}}{q{{\left({\frac{\epsilon}{2}} \right)}^{q-1}}\left\langle{\delta{\xi_0},\xi } \right\rangle }d\mu\left(\xi\right).
\end{split}
\end{equation}
Notice that $\lim\limits_{\delta \to 0^{+}}{\delta^{1-q}}=0$. Hence, there exists a small enough $\delta_{0}>0$ such that
${\Phi _Q}\left( {\zeta \left( Q \right) + \delta {\xi _0}} \right) > {\Phi _Q}\left( {\zeta \left( Q \right)} \right)$,
which leads to a contradiction, as $\zeta(Q)$ was chosen such that $\Phi_Q\left( \zeta(Q) \right) = \sup\limits_{\zeta \in Q} \Phi_Q(\zeta)$. Therefore, we conclude that $\zeta\left(Q\right)\in \text{int}Q$.

Thus, for any $x_0\in \mathbb{R}^{n}$, we have
\begin{equation*}
\begin{split}
{\Phi _{Q + x_0}}\left( {\zeta \left( {Q + x_0} \right)} \right)
&=\mathop {\sup }\limits_{\zeta  \in Q + x_0} \int_{{\mathbb{S}^{n-1}}} {{{\left( {{h_{Q + x_0}}
\left( \xi  \right) - \left\langle {\zeta,\xi } \right\rangle } \right)}^q}} d\mu \left( \xi  \right)\\
&=\mathop {\sup }\limits_{\zeta  \in Q} \int_{{\mathbb{S}^{n-1}}} {{{\left( {{h_Q}\left( \xi  \right)
- \left\langle {\zeta ,\xi } \right\rangle } \right)}^q}} d\mu \left( \xi  \right)\\
&={\Phi _Q}\left( {\zeta \left( Q \right)} \right)\\
&=\int_{{\mathbb{S}^{n-1}}} {{{\left( {{h_{Q + x_0}}\left( \xi  \right)
-\left\langle {{\zeta(Q)+ x_0},\xi } \right\rangle } \right)}^q}} d\mu \left( \xi  \right)\\
&={\Phi _{Q + x_0}}\left( {\zeta(Q)+ x_0} \right).
\end{split}
\end{equation*}
Therefore, by the uniqueness of the extreme point  $\zeta\left(Q+x_0\right)$, we conclude that ${\zeta\left(Q+x_0\right)}={\zeta\left(Q\right)}+x_0$.
\end{proof}

\begin{lemma}\label{lem:4.3}
Let $0<q<1$, $\mu$ be a finite Borel measure on $\mathbb{S}^{n-1}$, and
$\left\{Q_{j}\right\}^\infty_{j=1}$ be a sequence of compact convex sets.
If $Q_j$ converges to a compact convex set $Q$ as $j\to\infty$, then we have
$\lim\limits_{j\to \infty}\zeta\left(Q_{j}\right)=\zeta\left(Q\right)$ and $\mathop {\lim }\limits_{j \to \infty }\Phi_{Q_{j}}\left(\zeta \left(Q_{j}\right)\right)=\Phi_{Q}\left(\zeta \left(Q\right)\right)$.
\end{lemma}

\begin{proof}
Since the sequence $\{\zeta(Q_j)\}$ is bounded, there exists a convergent subsequence
(still denoted by $\{\zeta(Q_j)\}$) that converges to some $\zeta_{0}\in Q$.

Next, we prove that $\zeta_{0}=\zeta(Q)$. If otherwise, by using \eqref{4.4}
and Lemma \ref{lem:4.2}, we have
\[\mathop {\lim }\limits_{j \to \infty } {\Phi _{{Q_j}}}\left( {\zeta \left( {{Q_j}} \right)} \right)
= {\Phi _Q}\left( {{\zeta _0}} \right)<{\Phi _Q}\left( {{\zeta \left(Q\right)}}\right)
=\mathop {\lim }\limits_{j \to \infty } {\Phi _{{Q_j}}}\left( {\zeta \left( {{Q}} \right)} \right).\]
On the other hand, since ${\zeta \left( {{Q}} \right)}\in \text{int}Q_{j}$ for sufficiently large $j$, it follows that
${\Phi _{{Q_j}}}\left( {\zeta \left( {{Q_j}} \right)} \right)
>{\Phi _{{Q_j}}}\left( {\zeta \left( Q \right)} \right)$
for sufficiently large $j$. This contradiction implies that
$\zeta_{0}=\zeta\left( {{Q}} \right)$. Using \eqref{4.4} again,
we can verify that
$\mathop {\lim }\limits_{j \to \infty }\Phi_{Q_{j}}\left(\zeta \left(Q_{j}\right)\right)
=\Phi_{Q}\left(\zeta \left(Q\right)\right)$.
\end{proof}

Now, we are able to prove the Theorem \ref{th:1.1} as follows.

\begin{proof}[Proof of Theorem \ref{th:1.1}]
Recall that the Wulff shape $K_f$ associated with a function $f\in C_+\left(\mathbb{S}^{n-1}\right)$
is given by
\begin{equation}\nonumber
{K_f}
=\left\{{x\in\mathbb{R}^{n}:\left\langle {x,u}\right\rangle
\le f\left(u\right)}\
\text{for all}\ u\in\mathbb{S}^{n-1}\right\}.
\end{equation}
Then for $0<q<1$, $f\in C_+\left(\mathbb{S}^{n-1}\right)$, and a finite Borel measure $\mu$ on $\mathbb{S}^{n-1}$, we introduce a functional $\Phi_f:K_f\to\mathbb{R}$ by
\begin{equation}\label{4.5}
{\Phi_f}\left(\zeta\right)
=\int_{{\mathbb{S}^{n-1}}} {{{\left( {f\left(\xi\right)-\left\langle {\zeta,\xi}\right\rangle}\right)}^{q}}}d\mu\left(\xi\right),
\end{equation}
for $\zeta\in K_f$. We then construct the following minimization problem:
\begin{equation}\label{4.6}
\mathop {\inf }\limits_{f \in {C_+}
\left( {{\mathbb{S}^{n-1}}} \right)}
\left\{ {\mathop {\sup }\limits_{\zeta  \in {K_f}} {\Phi _{f}}
\left( \zeta  \right):\Gamma \left( K_{f} \right)= \Gamma \left( B_2^n \right)} \right\}.
\end{equation}

Since $h_{K_f}\le f$ and $K_{h_{K_f}}=K_f\in \mathcal K_o^n$ for any
$f\in C_+\left(\mathbb{S}^{n-1}\right)$, by \eqref{4.4} and \eqref{4.5}, we obtain that
$$\Phi_{K_f}\left(\zeta\right)=\Phi_{h_{K_f}}\left(\zeta\right)\le{\Phi_f}\left(\zeta\right),$$
where $\zeta\in{K_f}$. It follows that
$\mathop{\sup}\limits_{\zeta\in{K_f}} {\Phi_{K_f}}\left(\zeta\right)
\le\mathop{\sup}\limits_{\zeta\in{K_f}} {\Phi_f}\left(\zeta\right).$
Therefore, we can search for the minimum for \eqref{4.6} among the support functions
of convex bodies that contain the origin in their interiors,
and we can verify that $h_K$ is a solution to \eqref{4.6} if and only if $K$
is a solution to the problem
\begin{equation}\label{4.7}
\mathop {\inf }\limits_{Q \in \mathcal K_o^n}
\left\{ {\mathop {\sup }\limits_{\zeta \in Q} {\Phi_Q}
\left( \zeta  \right):\Gamma \left( Q \right)=\Gamma \left(B_2^n\right)} \right\}.
\end{equation}
	
Let $\left\{{Q_j}\right\}_{j = 1}^\infty $ be a minimizing sequence for the problem \eqref{4.7}.
That is, $\Gamma\left(Q_j \right)=\Gamma\left(B_2^n \right)$ and
$$\mathop {\lim}\limits_{j\to\infty} {\Phi _{Q_j}}
\left( {\zeta \left( {Q_j} \right)} \right)
= \mathop {\inf }\limits_{Q \in \mathcal K_o^n}
\left\{ {\Phi_Q\left(\zeta \left( {Q} \right)\right):\Gamma \left( Q \right)
= \Gamma \left( B_2^n \right)} \right\}.$$
According to Lemma \ref{lem:4.2}, we can suitably translate each $Q_j$  to
obtain a sequence $\left\{K_j\right\}_{j=1}^\infty$
in $\mathcal K_o^n$
such that $\zeta\left(K_j\right)=o$
and $\Gamma \left( {{K_j}} \right)=\Gamma \left( B_2^n \right)$ by \eqref{3.2}.
Therefore, $\left\{ {{K_j}} \right\}_{j = 1}^\infty$ is also the minimizing sequence for the problem \eqref{4.7}, and ${\Phi _{K_j}}\left(o\right)$ converges to
$$\mathop {\inf }\limits_{Q \in {\mathcal{K}_o^n}}
\left\{ {{\Phi_Q}\left( {\zeta \left( Q \right)} \right):\Gamma \left( Q \right)
= \Gamma \left( B_2^n \right)} \right\},$$
as $j\to\infty$.	
	
We now prove that the sequence $\left\{{K_j} \right\}$ is uniformly bounded.
To do so, we let
$R_j:=\mathop {\max }\limits_{\xi  \in {\mathbb{S}^{n-1}}} {h_{{K_j}}}\left(\xi\right)$
and assume that the maximum can be achieved by some $\xi_0\in \mathbb{S}^{n-1}$.
Then, we have
$$R_{j}{\left\langle {{\xi _0},\xi } \right\rangle _ + } \le {h_{{K_j}}}\left( \xi  \right)$$
for all $j$ and $\xi\in\mathbb{S}^{n-1}$, and hence
\begin{equation}\label{4.8}
\int_{\mathbb{S}^{n-1}}
{{\left({R_{j}{{\left\langle{{\xi _0},\xi} \right\rangle}_+}}\right)}^q}d\mu
\left(\xi\right)
\le\int_{\mathbb{S}^{n-1}}{{\left( {{h_{{K_j}}}\left( \xi  \right)} \right)}^q} d\mu \left( \xi\right)
={\Phi_{K_j}}\left(o\right).
\end{equation}
On the other hand, for sufficiently large $j$, we have
\begin{equation}\nonumber
{\Phi_{K_j}}\left(o\right)
\le\Phi_{B_2^n-\zeta\left(B_2^n\right)}\left(o\right)
=\int_{\mathbb{S}^{n-1}}{{\left(1-\left\langle{\zeta\left(B_2^n\right),\xi} \right\rangle\right)}^q} d\mu \left( \xi\right).
\end{equation}
This, together with \eqref{4.8}, implies that $\{R_j\}$ is uniformly bounded,
where we have used the fact that the measure $\mu$ is finite and not concentrated on
any closed hemisphere. Therefore, the boundedness of the sequence $\left\{{K_j}\right\}$
follows.
By the Blaschke selection theorem, there exists a subsequence (still denoted by
$\left\{ {{K_j}} \right\}$) that converges to some compact convex set $\Omega$ as
$j \to \infty$. In the following, we prove that $\dim (\Omega)=n$.
If $\dim (\Omega)<n-1$, then
$\mathcal{H}^{n-1}\left(\Omega\right)=0=\mathcal{H}^{n-1}\left(\partial \Omega\right)$. It follows from definition \eqref{3.1} and Lemma \ref{lem:2.1} that $\Gamma \left( \Omega \right)=0$, which contradicts to the following
\begin{equation}\label{4.9}
\Gamma\left(\Omega\right)
=\lim\limits_{j\to\infty}\Gamma\left(K_j \right)
=\Gamma\left(B_2^n\right)
>0.
\end{equation}
If $\dim (\Omega) = n - 1$, there are at least two half-spaces containing $\Omega$ that
share a common boundary, and $\Omega$ degenerates to a $1$-codimensional
subset of a hyperplane.
By Lemma \ref{lem:2.1} again,
\begin{equation}\nonumber
\left| {\nabla u} \right|\le M,
\end{equation}
thus we obtain that	
\begin{equation}\nonumber
\Gamma\left(\Omega\right)
=\int_{\mathbb{S}^{n-1}}{h_\Omega}d{\mu_\Omega}
\le{M^{\mathbf{p}-1}}\int_{\mathbb{S}^{n-1}} {{h_\Omega}} d{S_\Omega}\left(\xi\right)
=0,
\end{equation}
which again contradicts to \eqref{4.9}. Therefore, $\text{dim}(\Omega)=n$ and $\Omega$ is indeed a convex body. By Lemma \ref{lem:4.3}, we have ${\zeta\left(\Omega\right)=o}$ and
\begin{equation}\label{4.10}
{\Phi _{{h_\Omega}}}\left( o \right)
=\mathop {\inf }\limits_{f \in {C^ + }\left( {{\mathbb{S}^{n-1}}} \right)}
\left\{ {\mathop {\sup }\limits_{\zeta\in{K_f}} {\Phi _f}\left( \zeta  \right):
\Gamma \left( {{K_f}} \right)
= \Gamma \left( B_2^n \right)} \right\}.
\end{equation}

Let $\Omega_1$ be a compact convex set containing the origin and $\Omega^t$ be
the Wulff shape of ${\left({h_\Omega^q+t{h_{\Omega_1}^q}}\right)}^{\frac{1}{q}}$
for a small enough $t$, where
$$\lambda\left(t\right)
:={\left( {\frac{{\Gamma \left( B_2^n \right)}}{\Gamma \left(\Omega^t\right)}}\right)^{\frac{1}{n-\mathbf{p}+1}}}.$$
Here, we have used the condition that $\mathbf{p}\neq n+1$.	
Then, by equalities \eqref{3.1} and \eqref{3.8}, we can verify that
$\Gamma\left({\lambda\left(t\right){\Omega^t}}\right)=\Gamma\left(B_2^n\right)$.
In the following, we prove that
$\zeta\left(t\right):=\zeta\left({\lambda\left(t\right)\Omega^t}\right)$
is differentiable at $t=0$.
	
Let $\zeta=(\zeta_1,\zeta_2,\ldots,\zeta_n)$ and $F=(F_1,F_2,\ldots,F_n)$ be a vector-value function
from an open neighbourhood of the origin $\left({0,0,0,\ldots,0}\right)$ in $\mathbb{R}^{n+1}$ to $\mathbb{R}^{n}$,
where
$${F_i}\left( {t,{\zeta _1},{\zeta _2}, \ldots ,{\zeta _n}} \right)
=\int_{{\mathbb{S}^{n-1}}} {\frac{{{\xi _i}}}{{{{\left( {\lambda \left( t \right){h_{{\Omega^t}}}
\left( \xi  \right) - \left( {{\zeta _1}{\xi _1}+{\zeta _2}{\xi _2}+ \cdots +{\zeta _n}{\xi _n}} \right)} \right)}^{1 - q}}}}} d\mu \left( \xi  \right)$$
for $i=1,2,\ldots,n$.
As $\zeta(t)$ is an extreme point of $\Phi_{\lambda(t)\Omega^t}\left(\zeta\right)$ for $\zeta\in\lambda(t)\Omega^t$, it follows that $F_i(t,\zeta(t))=0$.
Then, two functions both
\begin{equation}\nonumber
{\left. {\frac{{\partial {F_i}}}{{\partial t}}} \right|_{\left( {t,{\zeta _1},{\zeta _2}, \ldots ,{\zeta _n}} \right)}} = \int_{{\mathbb{S}^{n-1}}} {\frac{{\left( {q - 1} \right){\xi _i}\left( {\lambda '\left( t \right){h_{\Omega^t}}\left( \xi  \right)
+\lambda\left(t\right){h_{\Omega^t}'}\left( \xi  \right)} \right)}}{{{{\left( {\lambda \left( t \right){h_{{\Omega^t}}}\left( \xi  \right) - \left( {{\zeta _1}{\xi _1}+{\zeta _2}{\xi _2}+ \cdots +{\zeta _n}{\xi_n}} \right)} \right)}^{2-q}}}}} d\mu \left( \xi  \right)
\end{equation}
and
\begin{equation}\nonumber
{\left. {\frac{{\partial {F_i}}}{{\partial {\zeta _j}}}} \right|_{\left( {t,{\zeta _1},{\zeta _2}, \ldots ,{\zeta _n}} \right)}} = \int_{{\mathbb{S}^{n-1}}} {\frac{{\left( {1 - q} \right){\xi _i}{\xi _j}}}{{{{\left( {\lambda \left( t \right){h_{{\Omega^t}}}\left( \xi  \right) - \left( {{\zeta _1}{\xi _1}+{\zeta _2}{\xi _2}+ \cdots +{\zeta _n}{\xi _n}} \right)} \right)}^{2-q}}}}} d\mu \left( \xi  \right)
\end{equation}
are all continuous on a small neighbourhood of $\left(0,0,0,\ldots,0\right)$, and
\begin{equation}\label{4.11}
{\left( {{{\left. {\frac{{\partial F}}{{\partial \zeta }}} \right|}_{\left( {0,0,0, \ldots ,0} \right)}}} \right)_{n \times n}} = \int_{{\mathbb{S}^{n-1}}} {\frac{{\left( {1 - q} \right)\xi^\intercal \xi}}{{h_\Omega^{2 - q}\left( \xi  \right)}}} d\mu \left( \xi  \right),
\end{equation}
where ${{\xi ^\intercal}\xi }$ is an $(n\times n)$ matrix.
	
As the measure $\mu$ is not concentrated on any closed hemisphere, for any
nonzero $x\in\mathbb{R}^{n}$, we have
\begin{equation}\nonumber
{x}{\left( {{{\left. {\frac{{\partial F}}{{\partial \zeta }}} \right|}_{\left( {0,0,0, \ldots ,0} \right)}}} \right)_{n \times n}}x^\intercal
= \int_{{\mathbb{S}^{n-1}}} {\frac{{\left( {1 - q} \right){{\left\langle {x,\xi } \right\rangle }^2}}}{{h_\Omega^{2 - q}\left( \xi  \right)}}} d\mu \left( \xi  \right)
>0.
\end{equation}
It follows that the matrix in \eqref{4.11} is positive definite. Then, by
${F_i}\left({0,0,0,\ldots,0}\right)=0$ and the continuity of ${\partial {F_i}}
\mathord{\left/{\vphantom{\partial{F_i} {\partial{\zeta _j}}}}\right.
\kern-\nulldelimiterspace} {\partial {\zeta _j}}$
on a neighbourhood of $\left({0,0,0,\ldots ,0}\right)$, one can use the implicit
function theorem to obtain that $\zeta\left( t \right)$ is continuously differentiable
on a small neighbourhood of $\left({0,0,0,\ldots,0}\right)$.
Hence, the derivative $\zeta'\left(0\right)$ of $\zeta \left(t\right)$ at $t=0$ exists.
	
Put $\Phi\left(t\right):
=\Phi_{\lambda \left( t \right){{\left(h_\Omega^q + th_{\Omega_1}^q
\right)}^{\frac{1}{q}}}}
\left( {\zeta\left( t \right)} \right)$,
then \eqref{4.10} shows that $\Phi\left(t\right)$
attains the minimal value at $t=0$.
Thus by \eqref{4.10} and
\begin{equation}\nonumber
\lambda'\left( 0 \right)
= - \frac{1}{\left( {n-\mathbf{p} + 1}\right)\Gamma \left(B_2^n\right)}
{\left.{\frac{d}{dt}}\right|_{t = 0}}\Gamma\left({\Omega^t}\right),
\end{equation}
we have the following calculation:
\begin{equation}\label{4.12}
\begin{split}
0=&{\left.{\frac{d}{dt}}\right|_{t = 0}}\Phi \left( t \right)\\
=&{\left.{\frac{d}{dt}}\right|_{t = 0}}\int_{{\mathbb{S}^{n-1}}}
{{\left({\lambda\left(t\right){{\left( {h_\Omega^q + t{h_{\Omega_1}^q}}\right)}^{\frac{1}{q}}}
-\left\langle{\zeta \left( t \right),\xi } \right\rangle} \right)}^{q}} d\mu\left(\xi\right)\\
=&q\int_{{\mathbb{S}^{n-1}}} {h_\Omega^{q - 1}\left( {\lambda '\left( 0 \right){h_\Omega}
+ \frac{1}{q}h_\Omega^{1 - q}h_{\Omega_1}^q -\left\langle {\zeta'\left(0\right),\xi } \right\rangle } \right)}
d\mu \left( \xi  \right)\\
=&q\int_{{\mathbb{S}^{n-1}}} {h_\Omega^{q - 1}\left( { - \frac{{{h_\Omega}}}{{\left( {n-\mathbf{p} + 1} \right)
\Gamma \left( B_2^n \right)}}{{\left. {\frac{d}{{dt}}} \right|}_{t = 0}}\Gamma \left({\Omega^t} \right)
+\frac{1}{q}h_\Omega^{1 - q}h_{\Omega_1}^q} \right)} d\mu \left( \xi\right)\\
&-q\int_{{\mathbb{S}^{n-1}}} {\left\langle {\zeta'\left(0\right),h_\Omega^{{q} - 1}\xi } \right\rangle } d\mu
\left( \xi  \right)\\
=&-\frac{q}{{n-\mathbf{p} + 1}}\int_{{\mathbb{S}^{n-1}}} {\frac{{h_\Omega^q}}{{\Gamma \left( B_2^n\right)}}
{{\left. {\frac{d}{{dt}}}\right|}_{t = 0}}\Gamma \left(\Omega^t\right)} d\mu\left( \xi  \right)
+\int_{{\mathbb{S}^{n-1}}} {h_{\Omega_1}^q} d\mu \left( \xi  \right)\\
&-q\left\langle {\zeta'\left(0\right),\int_{{\mathbb{S}^{n-1}}} {h_\Omega^{{p}-1}\xi }
d\mu \left( \xi  \right)}\right\rangle.
\end{split}
\end{equation}
Since ${\zeta\left(\Omega\right)=o}$ is an extreme point of $\Phi_\Omega\left(\zeta\right)$
for $\zeta\in \Omega$, we have
\begin{equation}\nonumber
\int_{{\mathbb{S}^{n-1}}} h_\Omega^{q-1}\xi d\mu(\xi)=o.
\end{equation}
This, together with \eqref{4.12} and Theorem \ref{th:3.1}, gives that
\begin{equation}\label{4.13}
\begin{split}
\int_{{\mathbb{S}^{n-1}}} {{h_{\Omega_1}^q}} d\mu
=&\frac{q}{{n-\mathbf{p} + 1}}
\int_{{\mathbb{S}^{n-1}}} {\frac{{h_\Omega^q}}{{\Gamma \left( B_2^n \right)}}
{{\left. {\frac{d}{{dt}}} \right|}_{t = 0}}\Gamma \left( {\Omega^t} \right)} d\mu \\
=&\int_{{\mathbb{S}^{n-1}}} {\frac{{h_\Omega^q}}{{\Gamma \left( B_2^n \right)}}
\int_{{\mathbb{S}^{n-1}}} {{h_{\Omega_1}^q}} d{\mu _{\Omega,q}}} d\mu \\
=&\int_{{\mathbb{S}^{n-1}}} {{h_{\Omega_1}^q}
\int_{{\mathbb{S}^{n-1}}} {\frac{{h_\Omega^q}}{{\Gamma \left( B_2^n \right)}}d\mu }}
d{\mu_{\Omega,q}}.
\end{split}
\end{equation}
		
For any $f\in C_{+}\left(\mathbb{S}^{n-1} \right)$, there exists a sequence of positive twice
continuously differentiable functions $\left\{ {{f_j}} \right\}_{j = 1}^\infty$ that converges to $f$.
Then for each $f_{j}$, Lemma \ref{lem:4.1} shows that there exists a convex body $L_j$ containing the origin in its interior and a constant $r_j>0$, such that $f_j^q=h^q_{L_j}-h^q_{r_jB_2^n}$.
Hence, by \eqref{4.13}, we have
\begin{equation}\label{4.14}
\int_{{\mathbb{S}^{n-1}}} {{h_{L_j}^q}} d\mu
=\int_{{\mathbb{S}^{n-1}}} {{h_{L_j}^q}\int_{{\mathbb{S}^{n-1}}}
{\frac{{h_\Omega^q}}{{\Gamma \left( B_2^n \right)}}d\mu }} d{\mu _{\Omega,q}},
\end{equation}
and similarly,
\begin{equation}\label{4.15}
\int_{{\mathbb{S}^{n-1}}} {{h_{r_jB_2^n}^q}} d\mu
=\int_{{\mathbb{S}^{n-1}}} {{h_{r_jB_2^n}^q}\int_{{\mathbb{S}^{n-1}}}
{\frac{{h_\Omega^q}}{{\Gamma \left( B_2^n \right)}}d\mu }} d{\mu _{\Omega,q}}.
\end{equation}
By subtracting \eqref{4.15} from \eqref{4.14} and using the approximate argument,
we conclude
\begin{equation}\nonumber
\int_{{\mathbb{S}^{n-1}}} {{f^q}} d\mu
=c\int_{{\mathbb{S}^{n-1}}} {{f^q}} d{\mu _{\Omega,q}},
\end{equation}
where
$$
c=\int_{{\mathbb{S}^{n-1}}} \frac{{h_\Omega^q}}{\Gamma \left( B_2^n \right)}d\mu .
$$
By the Riesz representation theorem, we have
$\mu=c{\mu_{\Omega,q}}$.
Furthermore, Lemma \ref{lem:3.1} and Definition
\ref{def:3.1} imply that the $L_q$ $\mathbf{p}$-harmonic measure is
positively homogeneous of degree $(n-\mathbf{p}+1-q)$,
then there exists a convex body $\tilde{\Omega}$
so that $\mu=\mu_{\tilde{\Omega},q}$, if $\mathbf{p}\neq n+1-q$.

We have completed the proof of Theorem \ref{th:1.1}.
\end{proof}

\vskip 5mm \noindent  {\bf Acknowledgements}  This paper is partially supported by the China Postdoctoral Science Foundation (No. 2024M761902), the National Natural Science Foundation of China (No. 12371060) and the Shaanxi Fundamental Science Research Project for Mathematics and Physics (No. 22JSZ012). The third author also received support from the Mathematical Sciences Institute at the Australian National University.

\vskip 2mm
\noindent Hai Li, {\small\tt lihai121455@163.com}\\
{Department of Mathematics and Statistics, Shaanxi Normal University, Xi'an, 710119, China}

\vskip 2mm
\noindent Longyu Wu, {\small\tt wulongyu66@gmail.com}\\
{Department of Mathematics and Statistics, Shaanxi Normal University, Xi'an, 710119, China}

\vskip 2mm
\noindent Baocheng Zhu, {\small\tt zhubaocheng814@163.com}\\
{Department of Mathematics and Statistics, Shaanxi Normal University, Xi'an, 710119, China}


\vskip 5mm




\begin{thebibliography}{99}
\addtolength{\itemsep}{-1.5ex}

\bibitem{AV22}
M. Akman, J. Gong, J. Hineman, J.L. Lewis and A. Vogel,
{\it The Brunn-Minkowski inequality and a Minkowski problem for nonlinear capacity},
Mem. Amer. Math. Soc. {\bf275} (2022), vi+115.

\bibitem{AM24}
M. Akman and S. Mukherjee,
{\it On the Minkowski problem for $p$-harmonic measures},
arXiv:2310.10247 (2024), https://arxiv.org/abs/2310.10247.

\bibitem{BZ12}
K. B\"or\"oczky, E. Lutwak, D. Yang and G. Zhang,
{\it The log-Brunn-Minkowski inequality},
Adv. Math. {\bf231} (2012), 1974-1997.

\bibitem{BZ13}
K. B\"or\"oczky, E. Lutwak, D. Yang and G. Zhang,
{\it The logarithmic Minkowski problem},
J. Amer. Math. Soc. {\bf26} (2013), 831-852.

\bibitem{CL22}
S. Chen, Y. Feng and W. Liu,
{\it Uniqueness of solutions to the logarithmic Minkowski problem in $\mathbb{R}^3$},
Adv. Math. {\bf411} (2022), Paper No. 108782.

\bibitem{CZ23}
S. Chen, S. Hu, W. Liu and Y. Zhao,
{\it On the planar Gaussian-Minkowski problem},
Adv. Math. {\bf435} (2023), part A, Paper No. 109351.

\bibitem{CL20}
S. Chen, Y. Huang, Q. Li and J. Liu,
{\it The $L_p$-Brunn-Minkowski inequality for $p<1$},
Adv. Math. {\bf368} (2020), Paper No. 107166.

\bibitem{C06}
W. Chen,
{\it $L_p$ Minkowski problem with not necessarily positive data},
Adv. Math. {\bf201} (2006), 77-89.

\bibitem{CY76}
S.Y. Cheng and S.T. Yau,
{\it On the regularity of the solution of the $n$-dimensional Minkowski problem},
Comm. Pure Appl. Math. {\bf29} (1976), 495-516.

\bibitem{CW06}
K.S. Chou and X.J. Wang,
{\it The $L_p$-Minkowski problem and the Minkowski problem in centroaffine geometry},
Adv. Math. {\bf205} (2006), 33-83.

\bibitem{CZ09}
A. Cianchi,  E. Lutwak, D. Yang and G. Zhang,
{\it Affine Moser-Trudinger and Morrey-Sobolev inequalities},
Calc. Var. PDE. {\bf36} (2009), 419-436.

\bibitem{CF10}
A. Colesanti and M. Fimiani,
{\it The Minkowski problem for torsional rigidity},
Indiana Univ. Math. J. {\bf59} (2010), 1013-1039.

\bibitem{CZ15}
A. Colesanti, K. Nystr\"om, P. Salani, J. Xiao, D. Yang and G. Zhang,
{\it The Hadamard variational formula and the Minkowski problem for $p$-capacity},
Adv. Math. {\bf285} (2015), 1511-1588.

\bibitem{CS03}
A. Colesanti and P. Salani,
{\it The Brunn-Minkowski inequality for $p$-capacity of convex bodies},
Math. Ann. {\bf327} (2003), 459-479.

\bibitem{CK15}
D. Cordero-Erausquin and B. Klartag,
{\it Moment measures},
J. Funct. Anal. {\bf268} (2015), 3834-3866.

\bibitem{DZ12}
J. Dou and M. Zhu,
{\it The two dimensional $L_p$ Minkowski problem and nonlinear equations with negative exponents},
Adv. Math. {\bf230} (2012), 1209-1221.

\bibitem{FY22}
N. Fang, S. Xing and D. Ye,
{\it Geometry of log-concave functions: the $L_p$ Asplund sum and the $L_p$ Minkowski problem},
Calc. Var. PDE. {\bf61} (2022), Paper No. 45.

\bibitem{FX23}
Y. Feng, S. Hu and L. Xu,
{\it On the $L_p$ Gaussian Minkowski problem},
J. Differ. Equ. {\bf363} (2023), 350-390.

\bibitem{G13}
E.I. Galakhov,
{\it Comparison principles for the $p$-Laplacian operator},
J. Math. Sci. (N.Y.) {\bf202} (2014), 825-834.

\bibitem{G06}
R.J. Gardner,
{\it Geometric Tomography},
second ed., Cambridge Univ. Press, New York, 2006, xxii+492.

\bibitem{GT01}
D. Gilbarg, N.S. Trudinger,
{\it Elliptic partial differential equations of second order},
Springer-Verlag, Berlin, 2001, xiv+517.

\bibitem{G07}
P. Gruber,
{\it Convex and discrete geometry},
Springer, Berlin, 2007, xiv+578.

\bibitem{GZ24}
L. Guo, D. Xi and Y. Zhao,
{\it The $L_p$ chord Minkowski problem in a critical interval},
Math. Ann. {\bf389} (2024), 3123-3162.

\bibitem{HS09}
C. Haberl and F. Schuster,
{\it Asymmetric affine $L_p$ Sobolev inequalities},
J. Funct. Anal. {\bf257} (2009), 641-658.

\bibitem{HZ18}
H. Hong, D. Ye and N. Zhang,
{\it The $p$-capacitary Orlicz-Hadamard variational formula and Orlicz-Minkowski problems},
Calc. Var. PDE. {\bf57} (2018), Paper No. 5.

\bibitem{HS04}
C. Hu, X. Ma and C. Shen,
{\it On the Christoffel-Minkowski problem of Firey's $p$-sum},
Calc. Var. PDE. {\bf21} (2004), 137-155.

\bibitem{HZ23}
Z. Hu and H. Li,
{\it On the existence of solutions to the Orlicz-Minkowski problem for torsional rigidity}, Arch. Math. {\bf120} (2023), 543-555.

\bibitem{HX15}
Y. Huang, J. Liu and L. Xu,
{\it On the uniqueness of $L_p$-Minkowski problems: the constant $p$-curvature case in $\mathbb{R}^3$},
Adv. Math. {\bf281} (2015), 906-927.

\bibitem{HZ16}
Y. Huang, E. Lutwak, D. Yang and G. Zhang,
{\it Geometric measures in the dual Brunn-Minkowski theory and their associated Minkowski problems},
Acta Math. {\bf216} (2016), 325-388.

\bibitem{HZ21}
Y. Huang, D. Xi and Y. Zhao,
{\it The Minkowski problem in Gaussian probability space},
Adv. Math. {\bf385} (2021), Paper No. 107769.

\bibitem{HZ05}
D. Hug, E. Lutwak, D. Yang and G. Zhang,
{\it On the $L_p$ Minkowski problem for polytopes},
Discrete Comput. Geom. {\bf33} (2005), 699-715.

\bibitem{J89}
D. Jerison,
{\it Harmonic measure in convex domains},
Bull. Amer. Math. Soc. (N.S.) {\bf21} (1989), 255-260.

\bibitem{J91}
D. Jerison,
{\it Prescribing harmonic measure on convex domains},
Invent. Math. {\bf105} (1991), 375-400.

\bibitem{J96}
D. Jerison,
{\it A Minkowski problem for electrostatic capacity},
Acta Math. {\bf176} (1996), 1-47.

\bibitem{JZ16}
H. Jian, J. Lu and G. Zhu,
{\it Mirror symmetric solutions to the centro-affine Minkowski problem},
Calc. Var. PDE. {\bf55} (2016), Paper No. 41.

\bibitem{L77}
J.L. Lewis,
{\it Capacitary functions in convex rings},
Arch. Ration. Mech. Anal. {\bf66} (1977), 201-224.

\bibitem{L06}
J.L. Lewis,
{\it Note on $p$-harmonic measure},
Comput. Methods Funct. Theory {\bf6} (2006), 109-144.

\bibitem{L13}
J.L. Lewis, K. Nystr\"om and A. Vogel,
{\it On the dimension of $p$-harmonic measure in space},
J. Eur. Math. Soc. {\bf15} (2013), 2197-2256.

\bibitem{LZ24}
C. Li and X. Zhao,
{\it Flow by Gauss curvature to the Minkowski problem of $p$-harmonic measure},
arXiv:2404.18757v2 (2024), https://arxiv.org/abs/2404.18757.

\bibitem{LH23}
H. Li and Z. Hu,
{\it On the polar Orlicz Minkowski type problem for the general mixed $\mathfrak{p}$-capacity},
J. Math. Anal. Appl. {\bf522} (2023), Paper No. 126925.

\bibitem{LZ24+}
N. Li, D. Ye and B. Zhu,
{\it The dual Minkowski problem for unbounded closed convex sets},
Math. Ann. {\bf388} (2024), 2001-2039.

\bibitem{LZ20}
N. Li and B. Zhu,
{\it The Orlicz-Minkowski problem for torsional rigidity},
J. Differ. Equ. {\bf269} (2020), 8549-8572.

\bibitem{LW20}
Q. Li, W. Sheng and X.J. Wang,
{\it Flow by Gauss curvature to the Aleksandrov and dual Minkowski problems},
J. Eur. Math. Soc. {\bf22} (2020), 893-923.

\bibitem{L88}
G. Lieberman,
{\it Boundary regularity for solutions of degenerate elliptic equations},
Nonlinear Anal. {\bf12} (1988), 1203-1219.

\bibitem{LX24}
Y. Liu, X. Lu, Q. Sun and G. Xiong,
{\it The logarithmic Minkowski problem in $\mathbb{R}^2$},
Pure Appl. Math. Q. {\bf20} (2024), 869-902.

\bibitem{LW13}
J. Lu and X.J. Wang,
{\it Rotationally symmetric solutions to the $L_p$-Minkowski problem},
J. Differ. Equ. {\bf254} (2013), 983-1005.

\bibitem{L93}
E. Lutwak,
{\it The Brunn-Minkowski-Firey theory. I. Mixed volumes and the Minkowski problem},
J. Differ. Geom. {\bf38} (1993), 131-150.

\bibitem{LZ24++}
E. Lutwak, D. Xi, D. Yang and G. Zhang,
{\it Chord measures in integral geometry and their Minkowski problems},
Comm. Pure Appl. Math. {\bf77} (2024), 3277-3330.

\bibitem {LZ02}
E. Lutwak, D. Yang and G. Zhang,
{\it Sharp affine $L_p$ Sobolev inequalities},
J. Differ. Geom. {\bf62} (2002), 17-38.

\bibitem{LZ04}
E. Lutwak, D. Yang and G. Zhang,
{\it On the $L_p$-Minkowski problem},
Trans. Amer. Math. Soc. {\bf356} (2004), 4359-4370.

\bibitem{M24}
E. Milman,
{\it A sharp centro-affine isospectral inequality of Szeg\"o-Weinberger type and the $L^p$-Minkowski problem},
J. Differ. Geom. {\bf127} (2024), 373-408.

\bibitem{R22}
L. Rotem,
{\it A Riesz representation theorem for functionals on log-concave functions},
J. Funct. Anal. {\bf282} (2022), Paper No. 109396.

\bibitem{S14}
R. Schneider,
{\it Convex Bodies: The Brunn-Minkowski theory},
Cambridge Univ. Press, Cambridge, 2014, xxii+736.

\bibitem{S18}
R. Schneider,
{\it A Brunn-Minkowski theory for coconvex sets of finite volume},
Adv. Math. {\bf332} (2018), 199-234.

\bibitem{S24}
R. Schneider,
{\it A weighted Minkowski theorem for pseudo-cones},
Adv. Math. {\bf450} (2024), Paper No. 109760.

\bibitem{S02}
A. Stancu,
{\it The discrete planar $L_0$-Minkowski problem},
Adv. Math. {\bf167} (2002), 160-174.

\bibitem{S03}
A. Stancu,
{\it On the number of solutions to the discrete two dimensional $L_0$-Minkowski problem},
Adv. Math. {\bf180} (2003), 290-323.

\bibitem{U03}
V. Umanskiy,
{\it On solvability of two-dimensional $L_p$-Minkowski problem},
Adv. Math. {\bf180} (2003), 176-186.

\bibitem{TX23}
J. Tao, G. Xiong and J. Xiong,
{\it The logarithmic Minkowski inequality for cylinders},
Proc. Amer. Math. Soc. {\bf151} (2023), 2143-2154.

\bibitem{T84}
P. Tolksdorf,
{\it Regularity for a more general class of quasilinear elliptic equations},
J. Differ. Equ. {\bf51} (1984), 126-150.

\bibitem{XZ23}
D. Xi, D. Yang, G. Zhang and Y. Zhao,
{\it The $L_p$ chord Minkowski problem},
Adv. Nonlinear Stud. {\bf23} (2023), Paper No. 20220041.

\bibitem{X20}
J. Xiao,
{\it Prescribing capacitary curvature measures on planar
convex domains},
J. Geom. Anal. {\bf30} (2020), 2225-2240.

\bibitem{XX19}
G. Xiong, J. Xiong and L. Xu,
{\it The $L_p$ capacitary Minkowski problem for polytopes},
J. Funct. Anal. {\bf277} (2019), 3131-3155.

\bibitem{YZ23}
J. Yang, D. Ye and B. Zhu,
{\it On the $L_p$ Brunn-Minkowski theory and the $L_p$ Minkowski problem for $C$-coconvex sets},
Int. Math. Res. Not. {\bf2023} (2023), 6252-6290.

\bibitem {Z99}
G. Zhang,
{\it The affine Sobolev inequality},
J. Differ. Geom. {\bf53} (1999), 183-202.

\bibitem{Z14}
G. Zhu,
{\it The logarithmic Minkowski problem for polytopes},
Adv. Math. {\bf262} (2014), 909-931.

\bibitem{Z15}
G. Zhu,
{\it The $L_p$ Minkowski problem for polytopes for $0<p<1$},
J. Funct. Anal. {\bf269} (2015), 1070-1094.

\bibitem{ZX20}
D. Zou and G. Xiong,
{\it The $L_p$ Minkowski problem for the electrostatic $\mathfrak{p}$-capacity},
J. Differ. Geom. {\bf116} (2020), 555-596.


\end{thebibliography}
\end{document}